\documentclass{article}
\usepackage{amsmath}
\usepackage{amssymb}
\usepackage{amsthm}
\usepackage{graphicx}
\counterwithin{equation}{section}
\newtheorem{theorem}{Theorem}[section]
\newtheorem{lemma}[theorem]{Lemma}
\newtheorem{corollary}[theorem]{Corollary}
\newtheorem{definition}[theorem]{Definition}
\newtheorem{proposition}[theorem]{Proposition}

\usepackage{mathtools}
\usepackage{indentfirst}
\DeclareMathOperator{\Conf}{Conf}
\DeclareMathOperator{\supp}{supp}

\usepackage[top=1.25in, bottom=1.25in, left=1.25in, right=1.25in]{geometry}
\title{{\Huge Fisher Zeros and its Critical Behavior for the Ising Model on Sierpinski Gasket}}
\author{Shaosong Liu}
\date{}
\begin{document}
	\maketitle
	\begin{abstract}
		In this paper, we provide a proof of the explicit formula for the partition function of the Ising model on the Sierpinski gasket. Additionally, we demonstrate the dynamic behavior of the zero distribution of the partition function when a sequence of appropriate regions is selected. Through this analysis, it is shown that the density of zeros around temperature $0$ exponentially diminishes at a rate of $\frac{1}{3^n}$, and the pressure function approaches the logarithmic function.
	\end{abstract}
	
	\section{Introduction}
	
	\subsection{Background}
	
	The Ising model emerges as a powerful tool in characterizing magnetic materials, unraveling intricate physical phenomena such as spontaneous magnetization and phase transitions. In their seminal work, Lee and Yang [1] first proposed their fundamental approach to phase transitions, which involves studying the zeros of the partition function of a statistical system as a function of a complex parameter. For finite graphs, the partition function has positive coefficients so there are no real zeros. However, in the thermodynamic limit, the zeros have the potential to pinch the real axis. The pinching points represent phase transition locations on the parameter axis, and the distribution of zeros around them can be linked to the critical properties of the system.
	
	\subsection{Ising Model}
	
	Assume we have a graph $\mathcal{G}$ representing a magnetic matter in a certain scale. The Ising model is to associate a spin variable $\sigma(v)=\pm 1$ to every vertex $v$ of the graph. Let $\mathcal{V}$ be the set of all vertices in $\mathcal{G}$ and $\mathcal{E}$ be the set of all edges. In this context, for $v, w \in \mathcal{V}$, we use $<v, w> \in \mathcal{E}$ to denote the oriented edge. The map $\sigma:\mathcal{V}\to \left\lbrace \pm 1 \right\rbrace $ represents a spin configuration on $\mathcal{G}$.
	
	This paper focuses on the ferromagnetic model, where the total energy or Hamiltonian of $\sigma$ is given by $$H(\sigma):=-J\sum_{<v,w>\in \mathcal{E}}\sigma(v)\sigma(w)$$ where $J$ is a positive constant.
	
	The Gibbs weight of a spin configuration $\sigma$ is defined as $$W(\sigma):=e^{-H(\sigma)/T}$$ where $T>0$ denotes the temperature.
	
	To define the partition function, we introduce the configuration space $\Conf (\mathcal{G})$. It is the space of all spin configurations on $\mathcal{G}$. The partition function is the total Gibbs weight of the space: $$Z_{\mathcal{G}}:=\sum_{\sigma \in \Conf (\mathcal{G})}W(\sigma).$$
	
	\subsection{Sierpinski Gasket}
	
	Now we introduce our the Sierpinski gasket. It is a fractal graph built recursively. Initially we have a triangle $\mathcal{G}_0$ with three vertices and three edges. Then for $n$th stage $\mathcal{G}_n$, it is built by joining three copies of $\mathcal{G}_{n-1}$ at their external corners(figure 1). As $n \to \infty$, we obtain our hierachical gasket.
	
	\begin{figure}[htbp]
		\centering
		\includegraphics[width=0.6\textwidth]{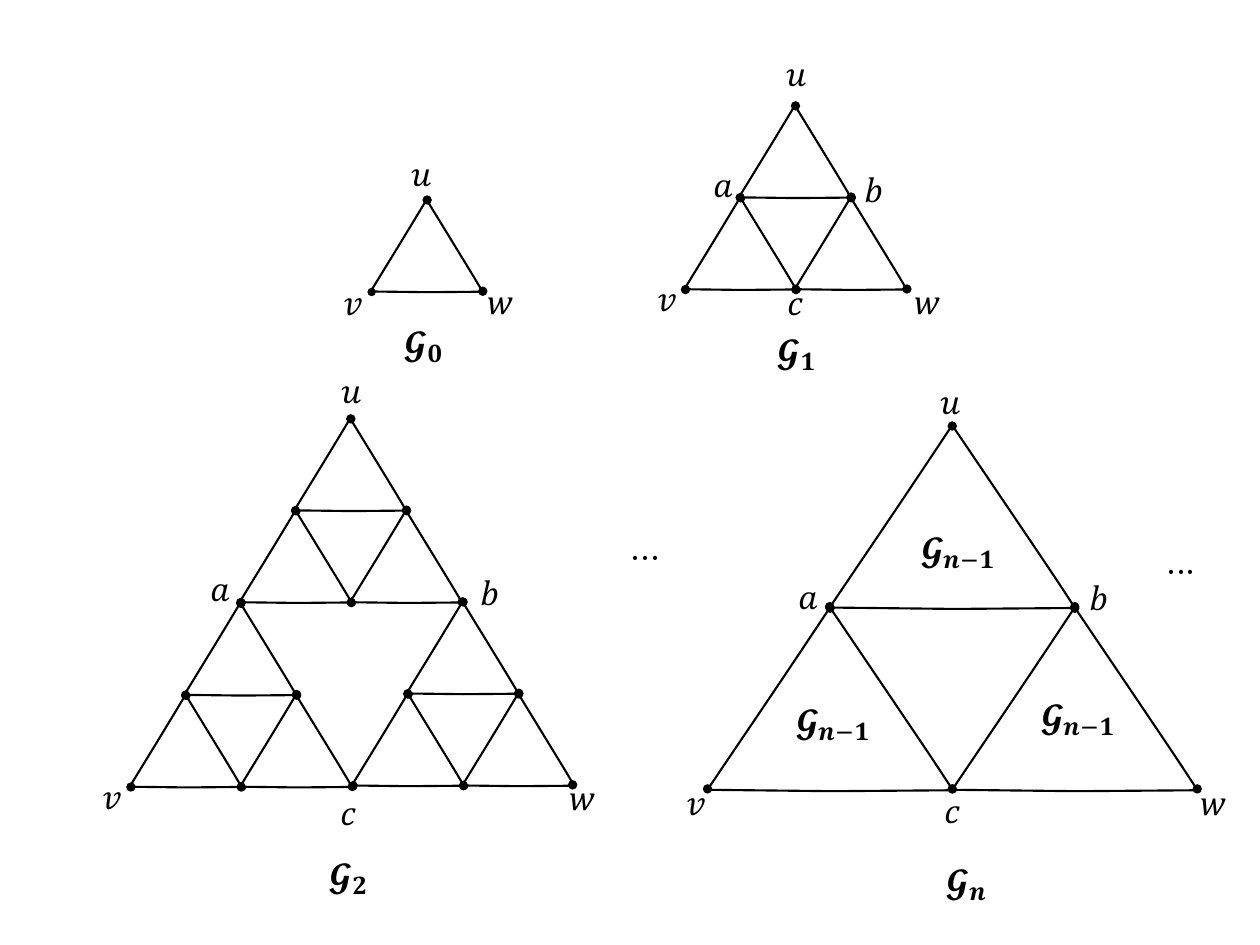}
		\caption{Gasket's construction}
		\label{Figure1}
	\end{figure}
	
	In the following, we write $W_n$ and $Z_n$ representing the Gibbs weight and partition function of $\mathcal{G}_n$ respectively in terms of the variable
	\begin{align*}
		y:=e^{\frac{J}{T}}.
	\end{align*} The pressure function is the limit of the logarithm of partition functions
	\begin{align}\label{def of pressure}
		p:=\lim_{n \to \infty}\frac{1}{4\cdot 3^n}\log|Z_n|.
	\end{align} When $p$ blows up near some tempature $T$, that is where the phase transition occurs.
	
	\subsection{Main results}
		
	Our goal is to illustrate the pressure function for Sierpinski gasket with Ising model. In [2], they show such partition functions satisfy a recurrence relation in terms of a certain multivalued map. Here we introduce a sequence of integer polynomials $\left\lbrace T_n \right\rbrace _{n=0}^{\infty}$ defined recursively in terms of a quadratic polynomial, from which we can factorize partition functions. In particular, we avoid multivalued maps entirely in our approach.
	
	Here is our main theorem.
	\begin{theorem}\label{pressure function and log}
		For $y \in \mathbb{R}^+:=[\, 0, \infty)$, the pressure function $p(y)$ is analytic and
		\begin{align*}
			\lim_{y \to \infty}\left| p(y)-\frac{3}{4}\log y\right| =0.
		\end{align*}
	\end{theorem}
	The pressure function also approaches to $\log $ if $y$ goes to $\infty$ along other directions on the complex plane $\mathbb{C}$.
	
	Our first step to achieve that is the expression of $Z_n$.
	
	\begin{theorem}\label{property of Mn}
		For every $n \in \mathbb{N}$ $\footnote {For this article $\mathbb{N}$ represents the set of all integers bigger than or equal to $0$, and $\mathbb{N}^+$ is $\mathbb{N} \setminus \left\lbrace 0\right\rbrace .$}$, the partition function $Z_n(y)$ satisfies that $y^{3^n}Z_n(y)$ is a polynomial in $\mathbb{Z}[y^4]$. Further, define
		\begin{align}\label{def of Mn}
			M_n(y^4):=\frac{y^{3^n}Z_n(y)}{2},
		\end{align}
		then $M_n(x)$ belongs to $\mathbb{Z}[x]$ and there is a recursive relation for $\left\lbrace M_n(x) \right\rbrace _{n=0}^{\infty}$:
		\begin{align*}
			M_n(x)=M_{n-1}(f(x))[(x+1)(x+3)]^{3^{n-1}}, 
		\end{align*}
		where $f(x):=\frac{x^2-x+4}{x+3}$.
	\end{theorem}
	
	For coefficients and the degree of $M_n$, we have:
	
	\begin{corollary}\label{coefficients of Mn}
		For every $n \in \mathbb{N}$, $M_n(x)$ is a monic polynomial of degree $3^n$ and has coefficients satisfying: \begin{align*}
			\frac{M_n(2x+1)}{2^{3^n}}\in \mathbb{Z}[x].
		\end{align*}The degree of $Z_n$ as a rational function of $y$ is $4\cdot 3^n$.
	\end{corollary}
	
	\noindent\textbf{Remark.} Here we give some initial expressions for $M_n$ and $Z_n$:
	\begin{equation}\label{examples of polynomials}
		\begin{aligned}
			M_0(x)&=x+3\\
			M_1(x)&=(x^2+2x+13)(x+1)\\
			M_2(x)&=(x^4+26x^2+72x+157)(x^2+7)(x+1)^3\\
			Z_0(y)&=\frac{2y^4+6}{y}\\
			Z_1(y)&=\frac{2(y^8+2y^4+13)(y^4+1)}{y^3}\\
			Z_2(y)&=\frac{2(y^{16}+26y^8+72y^4+157)(y^8+7)(y^4+1)^3}{y^9}.
		\end{aligned}
	\end{equation}
	
	It is not hard to guess that all coefficients(including the constant term) of $Z_n$ are real and positive so no phase transition for finite points. To observe the thermodynamic limit and phase transition near $y=\infty$, we consider \( Z_n \) as a map acting on \( \hat{\mathbb{C}} \), the Riemann sphere. After a certain coordinate change, we are able to deal with the normalized Laplacian of $Z_n$:
	\begin{align}\label{def of zetan}
		\zeta_n:=\frac{1}{4\cdot 3^n}\Delta \log (|Z_n|),
	\end{align} which plays a key role to compute the pressure function.
	
	\begin{theorem}\label{zeta converge}
		In the norm of total variation, $\left\lbrace \zeta_n \right\rbrace _{n=0}^{\infty}$ converges to $\zeta_{\infty}$ given by
		\begin{align*}
			\zeta_{\infty}:=-\frac{1}{4}\delta(0)-\frac{3}{4}\delta(\infty)+\sum_{j=0}^{\infty}\frac{1}{4\cdot 3^{j+1}}\left(  \sum_{\omega \in q^{-1}\circ f^{-j}(-1)}\delta(\omega)\right) 
		\end{align*}
		where $q(y):=y^4$ is the quartic map.
	\end{theorem}
	
	After Theorem \ref{zeta converge}, we can rewrite $p(y)$ like
	\begin{align*}
		p(y)=\int_{\mathbb{C}}\log |y-t|\mathrm{d} \zeta_{\infty}(t).
	\end{align*}
	
	By studying the inverse branch of \( f \), we can describe the zero distributions for \( M_n \) near \(\infty\) and prove that:
	
	\begin{theorem}\label{positive distance for zeros and real line}
		The closure of the set with all zeros of $\left\lbrace M_n(x) \right\rbrace _{n=0}^{\infty}$ has a positive distance to $\mathbb{R}^+$.
	\end{theorem}
	
	Then with the relationship between \( M_n \) and \( Z_n \) demonstrated in Theorem \ref{property of Mn}, we are able to prove Theorem \ref{pressure function and log}.
	
	\subsection{Outline of the paper}
	
	In Section 2, we demonstrate the recursive relation for \( Z_n \) step by step.
	
	In Section 3, we prove properties of $M_n$. During this process, we also introduce a new sequence of polynomials to illustrate zeros of $M_n$ more clearly.
	
	In Section 4, we prove Theorem \ref{zeta converge}. Additionally, we provide a corresponding result for the Laplacian of \( M_n \).
	
	In Section 5, we describe the asymptotic behavior of the inverse branch of \( f \). This allows us to control the distribution of preimages, which will be further illustrated in the subsequent section.
	
	In Section 6, we prove Theorem \ref{positive distance for zeros and real line}. To achieve this, we construct a partition for the set of zeros and use self-similarity under the translation map to find a positive gap.
	
	In Section 7, we compute the pressure function using the partition from Section 6 and prove Theorem \ref{pressure function and log}.
	
	\section{Recursive Relation}
	
	For the first partition function, we mark the three vertices as $u,v,w$ and write $Z_0$ as: $$\sum_{\sigma \in \Conf (\mathcal{G}_0)}W_0(\sigma).$$
	
	The term $W_0(\sigma)$ is actually a rank-3 tensor in $y^{\pm 1}$. Define two special configurations $\sigma_1$, $\sigma_2 $ like $\sigma_1(u)=\sigma_1(v)=\sigma_1(w)=1$, $\sigma_2(u)=\sigma_2(v)=1,\sigma_2(w)=-1$. They give two different values for $W_0$: $$W_0(\sigma_1)=y^3;\ W_0(\sigma_2)=y^{-1}.$$ By the reversing symmetry of the Ising model and reflection symmerty of the triangle, we know these two values cover all the eight possibilities of the Gibbs weight. The full partition function of $\mathcal{G}_0$ can be expressed like: $$Z_0(y)=2W_0(\sigma_1)+6W_0(\sigma_2)=2y^3+6y^{-1}.$$
	
	For the next stage, $\mathcal{G}_1$ is the union of three copies of $\mathcal{G}_0$, $W_1$ is also a rank-3 tensor in $W_0$(rank-9 in $y^{\pm 1}$). We mark the inner vertices of $\mathcal{G}_1$ by $a,b,c$ and $u,v,w$ represent exterior vertices. Fixing values of $\sigma(u),\sigma(v)$ and $\sigma(w)$ for $\sigma \in \Conf (\mathcal{G}_1)$, the Gibbs weight will be: $$W_1(\sigma)=W_0(\sigma_{uab})W_0(\sigma_{avc})W_0(\sigma_{bcw})$$ where $\sigma_{uab}=\sigma \bigg|_{uab} \in \Conf (\mathcal{G}_0)$ and the same for other two.
	
	Again by symmetry and values of $W_0$ we find that there are only two possible values for the sum of all $W_1(\sigma)$ of eight possible $\sigma$: $$\sum_{\mbox{\tiny$\begin{array}{c}
				\sigma \in \Conf (\mathcal{G}_1)\\
				\sigma(u)=\sigma(v)=\sigma(w)=1
			\end{array}$} }W_1(\sigma)= W_0(\sigma_1)^3+3W_0(\sigma_1)W_0(\sigma_2)^2+4W_0(\sigma_2) ^3$$
	and
	$$\sum_{\mbox{\tiny$\begin{array}{c}
				\sigma \in \Conf (\mathcal{G}_1)\\
				\sigma(u)=\sigma(v)=1,\sigma(w)=-1
			\end{array}$}}W_1(\sigma)=W_0(\sigma_1)^2W_0(\sigma_2)+4W_0(\sigma_1)W_0(\sigma_2)^2+3W_0(\sigma_2)^3.$$
	
	Recall for $\mathcal{G}_0$, $W_0(\sigma_1)$ and $W_0(\sigma_2)$ are the two possible values. Define $U_0=W_0(\sigma_1)$, $V_0=W_0(\sigma_2)$ and the above two sums as $U_1$, $V_1$ respectively, we can rewrite the two equations and our partition functions like:
	\begin{align*}
		U_1&=U_0^3+3U_0\cdot V_0^2+4V_0^3\\
		V_1&=U_0^2\cdot V_0+4U_0\cdot V_0^2+3V_0^3\\
		Z_0&=2U_0+6V_0\\
		Z_1&=2U_1+6V_1.
	\end{align*}
	We put their expressions in terms of $y$ in here for later use:
	\begin{equation}\label{initial values for UV}
		\begin{aligned}
			U_0(y)&=y^3\\
			V_0(y)&=y^{-1}\\
			U_1(y)&=y^9+3y+4y^{-3}\\
			V_1(y)&=y^5+4y+3y^{-3}.
		\end{aligned}
	\end{equation}
	We can find the expression of $Z_0$ and $Z_1$ in \eqref{examples of polynomials}.
	
	As Figure \ref{Figure1} shows, for general $n$, $W_n$ is again a rank-3 tensor in $W_{n-1}$. Inductively, we can prove its sum of Gibbs weights with fixed signs on three exterior vertices has only two possible values and can be computed by the same method. Therefore we define $U_n$ to be the one with same sign on exterior vertices and $V_n$ to be the other. In summary, for every $n \in \mathbb{N}^+$, there is a recursive relation between $U_n,V_n$ and $U_{n-1},V_{n-1}$:
	\begin{align}\label{U}
		U_n= U_{n-1}^3+3U_{n-1}\cdot V_{n-1}^2+4V_{n-1}^3
	\end{align}
	and
	\begin{align}\label{V}
		V_n=U_{n-1}^2\cdot V_{n-1}+4U_{n-1}\cdot V_{n-1}^2+3V_{n-1}^3
	\end{align}
	Our partition function of $\mathcal{G}_n$ is:
	\begin{align}\label{Zn}
		Z_n=2U_n+6V_n
	\end{align}
	
	\section{Factorization of the Partition function}
	This section discusses how to express $Z_n$ and its factorization under a certain coordinate. First we prove Theorem \ref{property of Mn}. Then we introduce
	\begin{definition}\label{def Tn}
		Define the elements of $\left\lbrace T_n(z) \right\rbrace _{n=0}^{\infty}$ in $\mathbb{Z}[z]$ recursively as follows:
		\begin{enumerate}
			\item[1.]
			$T_0(z):=z+1;$
			\item[2.]
			$T_n(z):=(z+2)^{3^{n-1}}\cdot T_{n-1}(g(z)) \text{ for } n\in \mathbb{N}^+$,
		\end{enumerate} where $g(z):=z^2+z$.
	\end{definition}
	
	The reason we bring in $T_n$ is that it has a nice factorization.
	
	\begin{theorem}\label{factor Tn}
		For every $n\in \mathbb{N}$, we have
		\begin{align*}
			T_n(z)=\left( g^{\circ n}(z)+1\right) \prod_{j=0}^{n-1}\left( g^{\circ j}(z)+2\right) ^{3^{n-j-1}}.
		\end{align*}
	\end{theorem}
	
	We also can write $M_n$ with the help of $T_n$.
	
	\begin{theorem}\label{Tn and degrees}
		For every $n \in \mathbb{N}$, the polynomial $T_n$ is monic with degree $3^n$ and has the relation with $M_n(x)$ as the form
		\begin{align}\label{Tn and Mn}
			M_n(x)=\frac{2^{\frac{3^{n+1}+1}{2}}}{\left( \phi(x)\right) ^{3^n}}T_n(\phi(x))
		\end{align}
		where $\phi(x):=\frac{4}{x-1}$ and it satisfies $g\circ \phi =\phi \circ f$.
	\end{theorem}
	Through those we get a clear view about zeros of $M_n$. At the end of this section, we prove Corollary \ref{coefficients of Mn}.
	
	\begin{proof}[Proof of Theorem \ref{property of Mn}]
		We prove our first claim by induction.
		
		Notice $$y U_0(y)=y^4,\ y V_0(y)=1$$ and $$y Z_0(y)=2y U_0(y)+6y V_0(y),$$ so it holds for $n=0$.
		
		Assume $y^{3^{n-1}}U_{n-1}(y)$ and $y^{3^{n-1}}V_{n-1}(y)$ are polynomials in terms of $y^4$ with integer coefficients, multiplying $y^{3^n}$ with \eqref{U} and \eqref{V} on both sides, we have:
		
		\begin{align*}
			y^{3^n}U_n(y)=\left( y^{3^{n-1}}U_{n-1}(y)\right) ^3+3\left( y^{3^{n-1}}U_{n-1}(y)\right) \cdot \left( y^{3^{n-1}}V_{n-1}(y)\right) ^2+4\left( y^{3^{n-1}}V_{n-1}(y)\right) ^3
		\end{align*}
		and
		\begin{align*}
			y^{3^n}V_n(y)&=\left( y^{3^{n-1}}U_{n-1}(y)\right) ^2\cdot \left( y^{3^{n-1}}V_{n-1}(y)\right) +4\left( y^{3^{n-1}}U_{n-1}(y)\right) \cdot \left( y^{3^{n-1}}V_{n-1}(y)\right) ^2\\
			&\qquad\qquad\qquad\qquad\qquad\qquad\qquad\qquad\qquad\qquad\qquad\ \ \  +3\left( y^{3^{n-1}}V_{n-1}(y)\right) ^3.
		\end{align*}
		So the assumption also holds for $n$. We proved the first claim.
		
		Now we can say for every $n \in \mathbb{N}$, there are polynomials $\tilde{U}_n$ and $\tilde{V}_n$ with integer coefficients satisfying:
		
		\begin{align*}
			\tilde{U}_n(y^4):=y^{3^n}\cdot U_n(y);\ \tilde{V}_n(y^4):=y^{3^n}\cdot V_n(y).
		\end{align*}
		Then by definitions of $Z_n$ and $M_n$(see \ref{Zn} and \ref{def of Mn}),we can express $M_n$ like:
		\begin{align*}
			M_n=\tilde{U}_n+3\tilde{V}_n.
		\end{align*}
		So $M_n(x)$ belongs to $\mathbb{Z}[x].$
		
		To get the recursion relation, we rewrite
		\begin{align*}
			M_n=\tilde{U}_n+3\tilde{V}_n=\tilde{V}_n\cdot \left( \frac{\tilde{U}_n}{\tilde{V}_n}+3\right).
		\end{align*} Through \eqref{initial values for UV}, the initial values for $\frac{U_n}{V_n}$ are:
		\begin{align*}
			\frac{U_0}{V_0}=y^4;\ \frac{U_1}{V_1}=\frac{y^8-y^4+4}{y^4+3}.
		\end{align*}
		We can see the natural map between these two is $f(y^4)$.
		
		What we want to prove next is
		\begin{align}\label{iteration}
			\frac{\tilde{U}_n}{\tilde{V}_n}=\frac{\tilde{U}_{n-1}}{\tilde{V}_{n-1}}\circ f.
		\end{align}
		Assume it holds for $n-1$, we start from the left handside:
		\begin{align*}
			\frac{\tilde{U}_n}{\tilde{V}_n}&=\frac{(\tilde{U}_{n-1})^3+3\tilde{U}_{n-1}\cdot (\tilde{V}_{n-1})^2+4(\tilde{V}_{n-1})^3}{(\tilde{U}_{n-1})^2\cdot \tilde{V}_{n-1}+4\tilde{U}_{n-1}\cdot (\tilde{V}_{n-1})^2+3(\tilde{V}_{n-1})^3}\\
			&=\frac{(\tilde{U}_{n-1})^2-\tilde{U}_{n-1}\cdot \tilde{V}_{n-1}+4(\tilde{V}_{n-1})^2}{\tilde{U}_{n-1}\cdot \tilde{V}_{n-1}+3(\tilde{V}_{n-1})^2}\\
			&=\frac{\left( \frac{\tilde{U}_{n-1}}{\tilde{V}_{n-1}}\right) ^2-\frac{\tilde{U}_{n-1}}{\tilde{V}_{n-1}}+4}{\frac{\tilde{U}_{n-1}}{\tilde{V}_{n-1}}+3}\\
			&=\frac{\left( \frac{\tilde{U}_{n-2}}{\tilde{V}_{n-2}}\right) ^2-\frac{\tilde{U}_{n-2}}{\tilde{V}_{n-2}}+4}{\frac{\tilde{U}_{n-2}}{\tilde{V}_{n-2}}+3}\circ f~(by~induction)\\
			&=\frac{(\tilde{U}_{n-2})^2-\tilde{U}_{n-2}\cdot \tilde{V}_{n-2}+4(\tilde{V}_{n-2})^2}{\tilde{U}_{n-2}\cdot \tilde{V}_{n-2}+3(\tilde{V}_{n-2})^2}\circ f\\
			&=\frac{\tilde{U}_{n-1}}{\tilde{V}_{n-1}}\circ f .
		\end{align*}
		Thus we proved \eqref{iteration}.
		
		With that, we can continue the computation of $M_n$:
		\begin{align*}
			M_n&=\tilde{V}_n\cdot \left( \frac{\tilde{U}_n}{\tilde{V}_n}+3\right) \\
			&=\tilde{V}_n\cdot \left( \frac{\tilde{U}_{n-1}}{\tilde{V}_{n-1}}\circ f+3\right) \\
			&=\frac{\tilde{V}_n}{\tilde{V}_{n-1}\circ f}\cdot \left[ \left( \tilde{U}_{n-1}+3\tilde{V}_{n-1}\right)\circ f \right] \\
			&=\frac{\tilde{V}_n}{\tilde{V}_{n-1}\circ f}\cdot \left( M_{n-1}\circ f \right).
		\end{align*}
		
		If we prove for every $n \in \mathbb{N}^+$, the first factor equals ${[(x+1)(x+3)]^{3^{n-1}}}$, we can say the recursive relation holds for our polynomial series.
		
		By observation, we find $$\frac{\tilde{V}_1 (x)}{\tilde{V}_0 (f(x))}=(x+1)(x+3),$$ so the following equation is sufficient: $$\frac{\tilde{V}_n}{\tilde{V}_{n-1}\circ f}=\left( \frac{\tilde{V}_{n-1}}{\tilde{V}_{n-2}\circ f}\right)^3.$$
		
		We compute it from the left handside:
		\begin{align*}
			\frac{\tilde{V}_n}{\tilde{V}_{n-1}\circ f}&=\frac{\tilde{U}_{n-1}^2\cdot \tilde{V}_{n-1}+4\tilde{U}_{n-1}\cdot \tilde{V}_{n-1}^2+3\tilde{V}_{n-1}^3}{\tilde{V}_{n-1}\circ f}\\
			&=\frac{\left( \frac{\tilde{U}_{n-1}}{\tilde{V}_{n-1}}\right) ^2+4\frac{\tilde{U}_{n-1}}{\tilde{V}_{n-1}}+3}{\frac{\tilde{V}_{n-1}\circ f}{\tilde{V}_{n-1}^3}}\\
			&=\frac{\left( \left( \frac{\tilde{U}_{n-2}}{\tilde{V}_{n-2}}\right) ^2+4\frac{\tilde{U}_{n-2}}{\tilde{V}_{n-2}}+3\right) \circ f}{\frac{\tilde{V}_{n-1}\circ f}{\tilde{V}_{n-1}^3}}\\
			&=\frac{\left( \tilde{U}_{n-2}^2\cdot \tilde{V}_{n-2}+4\tilde{U}_{n-2}\cdot \tilde{V}_{n-2}^2+3\tilde{V}_{n-2}^3\right) \circ f}{\frac{\left[ \tilde{V}_{n-1}\circ f\right] \cdot \left[ \tilde{V}_{n-2}^3\circ f\right] }{\tilde{V}_{n-1}^3}}\\
			&=\frac{\left[ \tilde{V}_{n-1}\circ f\right] \cdot \tilde{V}_{n-1}^3}{\left[ \tilde{V}_{n-1}\circ f\right] \cdot \left[ \tilde{V}_{n-2}^3\circ f\right] }\\
			&=\left( \frac{\tilde{V}_{n-1}}{\tilde{V}_{n-2}\circ f}\right)^3.
		\end{align*}
	\end{proof}
	
	Now we focus on \(M_n\), the numerator. To gain a clearer view of its zeros and simplify future iterations, we use the conjugation $\phi$ to transform \(f\) into the polynomial \(g\). After that we get the relationship between $T_n$ and $M_n$.
	\begin{proof}[Proof of Theorem \ref{Tn and degrees}]
		
		We prove \eqref{Tn and Mn} together with properties of $T_n$ by induction.
		
		Of course $T_0$ is monic with degree $3^0$ by definition. For $M_0(x)$, we rewrite it like: $$M_0(x)=x+3 = \frac{x+3}{x-1}\cdot (x-1) =\left( \frac{4+x-1}{x-1}\right)\cdot \left( \frac{x-1}{4}\right) \cdot 4 = \frac{4}{\phi(x)}\cdot T_0(\phi(x)).$$ So \eqref{Tn and Mn} holds for $M_0$ and $T_0$.
		
		Assume \eqref{Tn and Mn} holds for $n-1$ and $T_{n-1}$ is monic with degeree $3^{n-1}$, then by Definition \ref{def Tn}, the term $(z+2)^{3^{n-1}}$ makes $T_n$ a monic polynomial and we can count the degree of $T_n(z)$ as $3^{n-1}+2\cdot 3^{n-1}=3^n$.
		
		For $M_n$, we use Theorem \ref{property of Mn} to involve $T_{n-1}$:
		\begin{align*}
			M_n(x)&=M_{n-1}(f(x))[(x+1)(x+3)]^{3^{n-1}}\\
			&=\frac{2^{\frac{3^n+1}{2}}(f(x)-1)^{3^{n-1}}}{4^{3^{n-1}}}
			\cdot T_{n-1}(\phi (f(x)))\cdot [(x+1)(x+3)]^{3^{n-1}}.
		\end{align*}
		
		Now we compute the first factor. By pluging in $f(x)$, we get
		\begin{align*}
			\frac{(x-1)^{2\cdot 3^{n-1}}}{2^{\frac{3^{n-1}-1}{2}}(x+3)^{3^{n-1}}}.
		\end{align*}
		Eliminating common terms, we have $M_n(x)$ as:
		\begin{align*}
			M_n(x)=2^{\frac{1-3^{n-1}}{2}}(x-1)^{2\cdot 3^{n-1}}(x+1)^{3^{n-1}}\cdot T_{n-1}(\phi (f(x))).
		\end{align*}
		Next we create $\phi(x)$:
		\begin{align*}
			M_n(x)&=2^{\frac{1-3^{n-1}}{2}}\cdot \frac{4^{2\cdot 3^{n-1}}}{(\phi(x))^{2\cdot 3^{n-1}}}\cdot \left( \frac{4}{\phi(x)}+2\right)^{3^{n-1}} \cdot T_{n-1}(g(\phi(x)))\\
			&=\frac{2^{\frac{3^{n+1}+1}{2}}}{\left(\phi(x) \right) ^{3^n}}\cdot \left[ \left( 2+\phi(x)\right) ^{3^{n-1}}\cdot T_{n-1}(g(\phi(x)))\right] \\
			&=\frac{2^{\frac{3^{n+1}+1}{2}}}{(\phi(x))^{3^n}}\cdot T_n(\phi(x)).
		\end{align*}
		We proved our assumption.
		
		For the commutative diagram about $\phi$, it is clear that
		\begin{align*}
			g\circ \phi (x)=\frac{16}{(x-1)^2}+\frac{4}{x-1}=\frac{4(x+3)}{(x-1)^2}
		\end{align*}
		and
		\begin{align*}
			\phi \circ f(x)=\frac{4}{\frac{x^2-x+4}{x+3}-1}=\frac{4(x+3)}{x^2-2x+1}=\frac{4(x+3)}{(x-1)^2}.
		\end{align*}
	\end{proof}
	
	Next we show factorization of $T_n$.
	\begin{proof}[Proof of Theorem \ref{factor Tn}]
		Likewise, we prove it by induction.
		
		For \(T_0\), it is indeed \(g^0(z)+1\). Assuming we have such equality for \(T_{n-1}\), then \(T_n\) will be:
		\begin{align*}
			T_n(z)&=(z+2)^{3^{n-1}}T_{n-1}(g(z))\\
			&=(g^{\circ (n-1)}(g(z))+1)(g^0(z)+2)^{3^{n-1}} \prod_{j=0}^{n-2}(g^{\circ j}(g(z))+2)^{3^{n-j-2}}\\
			&=(g^{\circ n}(z)+1)\prod_{j=0}^{n-1}(g^{\circ j}(z)+2)^{3^{n-j-1}}.
		\end{align*}
	\end{proof}
	
	At the end, we show properties for coefficients of $M_n$.
	\begin{proof}[Proof of Corollary \ref{coefficients of Mn}]
		For the degree of \(M_n\), by \eqref{Tn and Mn}, it is sufficient to show that \(T_n\) has a nonzero constant term. Considering the recursive relation of \(T_n\), it is clear that its constant term is nonzero. Determining the degree of the rational function \(Z_n\) is straightforward, as we can directly derive it from \eqref{def of Mn}.
		
		Now we deal with the leading term of \(M_n(x)\) and the properties of the coefficients of \(M_n(2x+1)\). From Theorem \ref{property of Mn} and the degree of \(M_n(x)\), we can write \(M_n(x+1)\) as follows:
		\begin{align*}
			M_n(x+1)&=\frac{2^{\frac{3^{n+1}+1}{2}}}{\left( \frac{4}{x-1+1}\right)^{3^n} }T_n\left( \frac{4}{x-1+1}\right) \\
			&=2^{\frac{3^{n+1}+1}{2}}\left( \frac{x^{2^n}}{4^{2^n}}g^{\circ n}\left( \frac{4}{x}\right) +\frac{x^{2^n}}{4^{2^n}}\right) \prod_{j=0}^{n-1}\left( \frac{x^{2^j}}{4^{2^j}}g^{\circ j}\left( \frac{4}{x}\right) +\frac{x^{2^j}}{4^{2^j-\frac{1}{2}}}\right) ^{3^{n-j-1}}\\
			&=2^{\frac{3^{n+1}+1}{2}}\left( \frac{x^{2^n}}{2^{2^{n+1}}}g^{\circ n}\left( \frac{4}{x}\right) +\frac{x^{2^n}}{2^{2^{n+1}}}\right) \prod_{j=0}^{n-1}\left( \frac{x^{2^j}}{2^{2^{j+1}}}g^{\circ j}\left( \frac{4}{x}\right) +\frac{x^{2^j}}{2^{2^{j+1}-1}}\right) ^{3^{n-j-1}}.
		\end{align*}
		In step 2 we distribute $\left( \frac{x}{4}\right) ^{3^n}$ into every factor polynomial with proper order ($2^{2n}$ for the first one and $2^{2j}$ for the rest).
		
		The polynomial $g(z)$ is of degree $2$ with no constant term, so the coefficient of the leading term is to compare the order of $2$ of the leading term in every factor polynomial. The order of $2$ for their denominators is
		\begin{align*}
			2^{n+1}+\sum_{j=0}^{n-1}(2^{n+1}-1)\cdot 3^{n-j-1}&=2^{n+1}+3^n\sum_{j=1}^{n-1}\frac{2^{j+1}-1}{3^{j+1}}\\
			&=2^{n+1}+3^n\cdot \left( 2-2\cdot \left( \frac{2}{3}\right)^n-\frac{1}{2}+\frac{1}{2}\cdot \left( \frac{1}{3}\right)^n  \right) \\
			&=2^{n+1}+\frac{3^{n+1}}{2}-2^{n+1}+\frac{1}{2}\\
			&=\frac{3^{n+1}+1}{2}.
		\end{align*}
		
		It is the same as the order of $2$ for the first factor. In that sense, we can distribute $2^{\frac{3^{n+1}+1}{2}}$ properly in every factor polynomial ($2^{n+1}$ for the first and $2^{m+1}-1$ for the rest) to express $M_n(x+1)$ as a product of monic polynomials:
		
		\begin{align*}
			M_n(x+1)=\left( x^{2^n}\cdot g^{\circ n}\left( \frac{4}{x}\right) +x^{2^n}\right)\prod_{j=0}^{n-1}\left( \frac{x^{2^j}}{2} \cdot g^{\circ j}\left( \frac{4}{x}\right) +x^{2^j}\right) ^{3^{n-j-1}}.
		\end{align*}
		It is clear that $M_n(x+1)$ is monic, so is $M_n(x)$.
		
		When we compare the coefficients between the leading term and other lower order terms, the following parttern holds: if the order of $x$ is $3^n-k$ where $0< k\leq 3^n$, its coefficient will have a new factor $\left( \frac{4}{2}\right)^k=2^k$ comparing to $1$. Then $\frac{M_n(2x+1)}{2^{3^n}}\in \mathbb{Z}[x]$.
	\end{proof}
	
	\section{Convergence Behavior of the Corresponding Measure}
	
	Recall from \eqref{def of zetan} in Section 1 that the normalized Laplacian for partition function $Z_n$ is
	\begin{align*}
		\zeta_n=\frac{1}{4\cdot 3^n}\Delta \log |Z_n|.
	\end{align*}
	In this section we prove Theorem \ref{zeta converge}, which states it converges in the sense of total variation.
	
	For later use, here we also introduce the normalized Laplacian associated to $M_n$:
	\begin{align}\label{definition of mu}
		\mu_n:=\frac{1}{3^n}\Delta \log |M_n|.
	\end{align}
	and its property:
	\begin{corollary}\label{mu converge}
		$\left\lbrace \mu_n \right\rbrace _{n=0}^{\infty}$ converges in the norm of total variation and its limit $\mu_{\infty}$ has the expression:
		\begin{align*}
			\mu_{\infty}:=\sum_{j=0}^{\infty}\frac{1}{3^{j+1}}\left(  \sum_{\omega \in f^{-j}(-1)}\delta(\omega)\right) -\delta(\infty).
		\end{align*}
	\end{corollary}
	
	For the beginning, we need to clarify the multiplicities for all distinct roots of $T_n$, which is helpful in the proof of $\zeta_{\infty}$.
	\begin{lemma}\label{multiplicity of Tn}
		\hfill
		\begin{enumerate}
			\item[1.]The polynomial $g^{\circ n}(z)+1$, as well as polynomials $g^{\circ j}(z)+2$ for $m\in \left\lbrace 0,1,\ldots ,n-1\right\rbrace $ only have simple roots.
			\item[2.] Every pair of distinct polynomials from item 1 have no common roots.
		\end{enumerate}
	\end{lemma}
	
	\begin{proof}
		For \(g^{\circ j}(z) + 2 = 0\) with \(j \in \left\{0, 1, \ldots, n-1\right\}\), if it had a multiple root \(z_1 \in \mathbb{C}\), its derivative would vanish at \(z_1\). By the chain rule, we compute the derivative as \(\left(g^{\circ j}(z) + 2\right)^{\prime} = \prod_{k=0}^{j-1} g^{\prime}(g^{k}(z))\). Thus, at least one factor must be zero at \(z_1\). We know \(g^{\prime}(z) = 2z + 1\) has a single root at \(-\frac{1}{2}\). In other words, \(g^{k}(z_1) = -\frac{1}{2}\) for at least one \(k \in \left\{0, 1, \ldots, j-1\right\}\). Notice the inequality \(z < z(z + 1) = g(z) < 0\) for \(-1 < z < 0\); by induction, we can prove \(\left\{g^{j}(-\frac{1}{2}) \right\}_{j=0}^{j-k}\) is an increasing sequence with negative terms starting from \(-\frac{1}{2}\). This implies we can't have \(g^{\circ j}(z_1) = g^{j-k}(-\frac{1}{2}) = -2\). In other words, \(g^{\circ j}(z) + 2 = 0\) has only simple roots for \(j \in \left\{0, 1, \ldots, n-1\right\}\). Since \(-1 < -\frac{1}{2}\), the same argument holds if \(2\) is replaced by \(1\). Therefore, \(g^{\circ n}(z) + 1\) has no repeated roots either.
		
		Now, consider Item 2. Assume \(g^{\circ j_1}(z_2) = g^{\circ j_2}(z_2) = -2\) for some \(j_1 < j_2 < n \in \mathbb{N}\) and \(z_2 \in \mathbb{C}\). We compute \(g^{\circ (j_1+1)}(z_2) = g(-2) = 2\) and all the future iterations give positive results, excluding \(-2\). This means \(g^{\circ j_2}(z_2)\) can't be \(-2\); the assumption is wrong, so there are no common roots between \(g^{\circ j_1}(z) + 2 = 0\) and \(g^{\circ j_2}(z) + 2 = 0\). Notice that we can't get \(-1\) either from \(g^{\circ j}(z_2) = -2\). Therefore, \(g^{\circ n}(z) + 1 = 0\) has no common roots with \(g^{\circ j}(z) + 2 = 0\) for any integer \(j\) between \(1\) and \(n-1\).
	\end{proof}
	
	Now we can show the convergence of $\zeta_n$.
	
	\begin{proof}[Proof of Theorem \ref{zeta converge}]
		At first, we define:
		\begin{align*}
			\tau_n:=\frac{1}{3^n}\Delta \log |T_n|.
		\end{align*}
		Then we can also express $\tau_n$ like: $$\tau_n(z)=\frac{1}{3^n}\sum_{i=1}^{3^n}\Delta \log |z-z_i|.$$
		In here $z_i$, $1\leq i \leq 3^n$ are all the roots of $T_n$ counting multiplicities.
		
		To understand the convergence of $\tau_n$, we need to focus on the distribution of zeros for $T_n$.From the factorization of $T_n(z)$ and Lemma \ref{multiplicity of Tn}, we divide its zeros into different groups and rewrite $\tau_n$: $$\tau_n=\sum_{j=0}^{n-1}\frac{1}{3^{j+1}}\left(  \sum_{\omega \in g^{-j}(-2)}\delta(\omega)\right) +\frac{1}{3^n} \sum_{\omega \in g^{-n}(-1)}\delta(\omega)-\delta(\infty).$$
		
		Then we use $\phi \circ q$ and the connection between $Z_n$ and $T_n$ to pull-back $\tau_n$ to $\zeta_n$. Through  \eqref{Tn and Mn} and \eqref{def of Mn} we can do the following computation:
		\begin{align}\label{expression of zetan}
			\zeta_n&=\frac{1}{4\cdot 3^n}\Delta \log |Z_n| \notag\\
			&=-\frac{1}{4}\left( \delta(0)-\delta(\infty)\right) -\frac{1}{4}(\phi \circ q)^*(\delta(0)-\delta(\infty))+\frac{1}{4}(\phi \circ q)^*\tau_n  \notag\\
			&=-\frac{1}{4}\left( \delta(0)-\delta(\infty)\right) -\frac{1}{4}(\phi \circ q)^*(\delta(0)-\delta(\infty)) \notag\\
			&\ \ \ \ +\sum_{j=0}^{n-1}\frac{1}{4\cdot 3^{j+1}}\left(  \sum_{\omega \in q^{-1}\circ f^{-j}(-1)}\delta(\omega)\right) +\frac{1}{4\cdot 3^n} \sum_{\omega \in q^{-1}\circ f^{-n}(-3)}\delta(\omega)-\frac{1}{4}(\phi \circ q)^*\delta(\infty) \notag\\
			&=-\frac{1}{4}\delta(0)-\frac{3}{4}\delta(\infty)+\sum_{j=0}^{n-1}\frac{1}{4\cdot 3^{j+1}}\left(  \sum_{\omega \in q^{-1}\circ f^{-j}(-1)}\delta(\omega)\right) +\frac{1}{4\cdot 3^n} \sum_{\omega \in q^{-1}\circ f^{-n}(-3)}\delta(\omega).
		\end{align}
		
		In the sense of total variation, for any $n \in \mathbb{N^+}$, we have
		\begin{align*}
			|\zeta_{n}-\zeta_{\infty}|&=\left| -\frac{1}{12}\sum_{j=n}^{\infty}\frac{1}{3^j}\left( \sum_{\omega \in q^{-1}\circ f^{-j}(-1)}\delta(\omega)\right) +\frac{1}{4\cdot 3^n}\sum_{\omega \in q^{-1}\circ f^{-n}(-3)}\delta(\omega) \right| \\
			&=\frac{1}{12}\sum_{j=n}^{\infty}4\cdot \left( \frac{2}{3}\right) ^j+\frac{4}{4}\cdot \left( \frac{2}{3}\right)^n \\
			&=2\cdot \left( \frac{2}{3}\right) ^n.
		\end{align*} It goes to zero when $n$ goes to $\infty$.
	\end{proof}
	
	\begin{proof}[Proof of Corollary \ref{mu converge}]
		With \eqref{def of Mn} we can compute
		\begin{align*}
			\zeta_n=\frac{1}{4\cdot 3^n}\Delta \log |Z_n|=-\frac{1}{4}\left( \delta(0)-\delta(\infty)\right) +\frac{1}{4}q^*\mu_n.
		\end{align*}
		Following the expression of $\zeta_n$ in \eqref{expression of zetan}, we get
		\begin{align*}
			\mu_n=\sum_{j=0}^{n-1}\frac{1}{3^{j+1}}\left(  \sum_{\omega \in f^{-j}(-1)}\delta(\omega)\right) +\frac{1}{3^n} \sum_{\omega \in f^{-n}(-3)}\delta(\omega)-\delta(\infty).
		\end{align*}
		Using the same way as $\zeta_n$, we have $\mu_n$ converge to $\mu_{\infty}$ in the sense of total variation.
	\end{proof}
	
	\section{Dynamical Performance of the Inverse Branch}
	
	In this section, we devise a method for selecting the inverse branch of $f^{-1}$ that tends to $\infty$ when iterated over a specific domain $\mathbb{H}_0$, where it is defined as
	\begin{align*}
		\mathbb{H}_0:=\left\lbrace x\in \mathbb{C}|\Re (x)>10\right\rbrace .
	\end{align*}
	We also establish fundamental properties of this inverse branch. In the ensuing theorem, $h$ denotes the unique holomorphic inverse branch of $f$ defined on $\hat{\mathbb{C}}\backslash [-15,1]$, satisfying $h(x)-x-4 \to 0$ as $x \to \infty$, as shown in Lemma \ref{def of h} later.
	
	In the following discussion, our $\log$ refers to the branch of the logarithm function defined on $\mathbb{C}\backslash (\,-\infty,0]$, with the value $0$ assigned to the point $1$.
	
	\begin{theorem}\label{iteration of h}
		For all $n\in \mathbb{N^+}$ and $x \in \mathbb{H}_0$, $h^{\circ n}(x)$ is well defined and belongs to $\mathbb{H}_0$.
		There is a positive number $C>0$ such that for every $x \in \mathbb{H}_0$ and $n \in \mathbb{N^+}$, the following inequality holds:
		\begin{align}\label{main estimate}
			\left| h^{\circ n}(x)-x-4n+4 \log \left( \frac{x+4n}{x}\right) \right| < \frac{C\log \Re(x)}{\Re(x)}.
		\end{align}
		In particular, for $x\in \mathbb{H}_0$, we have:
		\begin{align}\label{log of h}
			\lim_{n \to \infty}\Im(\log(h^{\circ n}(x)))=0.
		\end{align}
	\end{theorem}
	
	In this section, our objective is to establish the theorem. We begin by presenting necessary calculus results in Section 5.1. Subsequently, in Section 5.2, we outline the definition of the desired inverse branch of $f$ and provide the proof of its key properties.
	
	\subsection{Preliminaries}
	
	First we choose an inverse branch of the square map.
	\begin{lemma}
		Define
		\begin{align*}
			s:\mathbb{C}\backslash (\,-\infty,0]\longrightarrow \mathbb{C}
		\end{align*}
		as the function:
		\begin{align*}
			s(w)&:=r^{\frac{1}{2}}e^{i\frac{\theta}{2}}\\
			where\ w=re^{i\theta}\ for&\ some\ r >0, \theta \in (-\pi,\pi).
		\end{align*}
		Then $s$ is holomorphic and $\left( s\left( w\right) \right) ^2=w$ for $w\in \mathbb{C}\backslash (\,-\infty,0]$.
		
	\end{lemma}
	\begin{proof}
		The domain $\mathbb{C}\backslash (\,-\infty,0]$ is simply connected so $s$ is holomorphic on it. It is clear that
		\begin{align*}
			\left( s(re^{i\theta})\right) ^2=\left( r^{\frac{1}{2}}e^{i\frac{\theta}{2}}\right) ^2=re^{i\theta}.
		\end{align*}
	\end{proof}
	
	We introduce a new function to prepare the branch of $f^{-1}$ which has connection with $s$.
	
	\begin{lemma}\label{def of a}
		Define a quadratic function like:
		\begin{align*}
			\mathfrak{a}(x):=1-\frac{64}{(x+7)^2}.
		\end{align*}
		If $x$ belongs to $\mathbb{C}\backslash [-15,1]$, $\mathfrak{a}(x)$ would locate in $\mathbb{C}\backslash (\,-\infty,0]$.
		
		Further, if $x$ belongs to $\mathbb{H}_0$,  then $\Re(\mathfrak{a}(x))>0$ and between norms there is:
		\begin{align*}
			\left| \mathfrak{a}(x)\right| >\mathfrak{a}(10)=\frac{225}{289}.
		\end{align*}
	\end{lemma}
	
	\begin{proof}
		Assume $\mathfrak{a}(x)=-r$ for some $r\geq 0$, then we get
		\begin{align*}
			1+r=\frac{64}{(x+7)^2}.
		\end{align*}
		So $x$ is real and $(x+7)^2\leq 64$, which implies $x \in [-15,1]$. We proved the first claim.
		
		For $x \in \mathbb{H}_0$, we know
		\begin{align*}
			\left| \frac{64}{(x+7)^2}\right| <\frac{64}{(10+7)^2}<1,
		\end{align*}
		then
		\begin{align*}
			\Re\left( \frac{64}{(x+7)^2}\right) <1 \Longrightarrow \Re(\mathfrak{a}(x))>0.
		\end{align*}
		For the norm,
		\begin{align*}
			|\mathfrak{a}(x)|\geq 1-\frac{64}{|x+7|^2}>1-\frac{64}{(10+7)^2}=\mathfrak{a}(10)=\frac{225}{289}.
		\end{align*}
	\end{proof}
	
	We also need some calculus results to state further properties of $f^{-1}$.
	
	\begin{lemma}\label{harmonic lemma}
		For any $x \in \mathbb{C}$ with $\Re(x)>0$, $n \in \mathbb{N^+}$ and positive number $p$, the finite sum of harmonic series have the following estimate:
		\begin{align}\label{harmonic}
			\left| \sum_{j=0}^{n-1}\frac{1}{x+pj}-\frac{1}{p}\log\left(\frac{x+pn}{x} \right) \right| <\sum_{j=0}^{n-1}\frac{p}{|x+pj|^2}.
		\end{align}
		Further, when $x\in \mathbb{H}_0$ and $1\leq p\leq 10$, for every $n\in \mathbb{N^+}$ the series also satisfies:
		\begin{align}\label{harmonic and log}
			\left| \sum_{j=0}^{n-1}\frac{1}{x+pj}\right| <\frac{2}{p}\left| \log(x+pn) \right| .
		\end{align}
		
		Similarly, if $x>10$ and $1\leq p\leq 6$, for all $n\in \mathbb{N^+}$, the following inequalies hold:
		\begin{align}
			\sum_{j=0}^{n-1}\frac{1}{(x+pj)^3}&<\frac{10}{px^2}\text{ , } \label{cubic inequality}\\
			\sum_{j=0}^{n-1}\frac{1}{(x+pj)^2}&<\frac{10}{px}\text{ , } \label{quadratic inequality}\\
			\sum_{j=0}^{n-1}\frac{\log (x+pj)}{(x+pj)^2}&<\frac{10\log x}{px}. \label{log quadratic inequality}
		\end{align}
	\end{lemma}
	
	\begin{proof}
		To prove \eqref{harmonic}, we rewrite the left side as:
		\begin{align*}
			\left| \sum_{j=0}^{n-1}\frac{1}{x+pj}-\frac{1}{p}\log\left(\frac{x+pn}{x} \right) \right| &=\left| \sum_{j=0}^{n-1} \frac{1}{x+pj}-\int_{0}^{n}\frac{1}{x+pl} \, \mathrm{d} l  \right| \\
			&< \sum_{j=0}^{n-1}\left| \frac{1}{x+pj}-\int_{j}^{j+1}\frac{1}{x+pl} \, \mathrm{d} l\right|  \\
			&=\sum_{j=0}^{n-1}\left| \int_{j}^{j+1}\frac{p(l-j)}{(x+pj)(x+pl)} \, \mathrm{d} l\right| \\
			&<\sum_{j=0}^{n-1}\frac{p}{\left| x+pj\right| ^2} .
		\end{align*}
		
		To deal with \eqref{harmonic and log}, when $x\in \mathbb{H}_0$ and $1\leq p\leq 10$, it is clear that
		\begin{align*}
			\sum_{j=0}^{n-1}\frac{p}{\left| x+pj\right| ^2}<\sum_{j=1}^{\infty}\frac{1}{pj^2}=\frac{\pi^2}{6p}.
		\end{align*}
		Then by triangle inequality
		\begin{align*}
			\left| \sum_{j=0}^{n-1}\frac{1}{x+pj}\right| <\frac{1}{p}\left| \log\left(\frac{x+pn}{x} \right) \right| +\frac{\pi^2}{6p}<\frac{1}{p}\left| \log(x+pn) \right|+\frac{2}{p}<\frac{2}{p}\left| \log(x+pn) \right| .
		\end{align*}
		
		For the next three ineuqalities, we prove the last inequality only. After that, the previous two can be addressed in the same manner.
		
		As we did to prove \eqref{harmonic}, we can give the following estimate:
		\begin{multline*}
			\left| \sum_{j=0}^{n-1}\frac{\log (x+pj)}{(x+pj)^2}+\frac{1}{p}\left( \frac{1+\log (x+pn)}{x+pn} -\frac{1+\log x}{x}\right) \right| \\
			<\sum_{j=0}^{n-1}\int_{j}^{j+1}\frac{(x+pl)^2\log (x+pj)-(x+pj)^2\log (x+pl)}{(x+pj)^2(x+pl)^2} \, \mathrm{d} l.
		\end{multline*}
		The numerator can be computed as:
		\begin{align*}
			(x+pl)^2\log (x+pj)-(x+pj)^2\log &(x+pl)\\
			&=(x+pj)^2\log \left( \frac{x+pj}{x+pl}\right) +[p^2(l-j)^2+2xp(l-j)]\log (x+pj)\\
			&<(p^2+2xp)\log (x+pj).
		\end{align*}
		So
		\begin{align*}
			\left| \sum_{j=0}^{n-1}\frac{\log (x+pj)}{(x+pj)^2}+\frac{1}{p}\left( \frac{1+\log (x+pn)}{x+pn} -\frac{1+\log x}{x}\right) \right| &<\sum_{j=0}^{n-1}\int_{j}^{j+1}\frac{(p^2+2xp)\log (x+pj)}{(x+pj)^2(x+pl)^2} \, \mathrm{d} l \\
			&<\sum_{j=0}^{n-1}\frac{3p\log (x+pj)}{2(x+pj)^3} \\
			&<\frac{9}{10}\sum_{j=0}^{n-1}\frac{\log (x+pj)}{(x+pj)^2}.
		\end{align*}
		Then we get
		\begin{align*}
			\sum_{j=0}^{n-1}\frac{\log (x+pj)}{(x+pj)^2}<10\left| \frac{1}{p}\left( \frac{1+\log (x+pn)}{x+pn} -\frac{1+\log x}{x}\right) \right| <\frac{10\log x}{px}.
		\end{align*}
	\end{proof}
	
	\subsection{Properties of the Inverse Branch}
	
	The inverse function $h$ is defined by the following.
	
	\begin{lemma}\label{def of h}
		There exists a unique branch of $f^{-1}$, denoted as $h$, which satisfies the following properties: the domain of $h$ is $\hat{\mathbb{C}}\backslash [-15,1]$ and it approaches the translation map $x+4$:
		\begin{align*}
			h(x)-x-4 \longrightarrow 0
		\end{align*}
		as $x \to \infty$ in its domain.
	\end{lemma}
	
	\begin{proof}
		Due to properties of $\mathfrak{a}$, we can define $h$ in $\mathbb{C}\backslash [-15,1]$ like:
		\begin{align*}
			h(x):=\frac{x+1+(x+7)\cdot s\circ \mathfrak{a}(x) }{2}
		\end{align*}
		so $h$ is holomorphic in $\mathbb{C}\backslash [-15,1]$.
		
		For $x=\infty$ on the Riemann sphere, if we define $h(\infty)=\infty$, we get $h$ is continuous on $\infty$. Thus on the domain $\hat{\mathbb{C}}\backslash [-15,1]$, we can state $h$ is holomorphic.
		
		Recall from Theorem \ref{property of Mn} that
		\begin{align*}
			f(x)=\frac{x^2-x+4}{x+3}.
		\end{align*} We verify the equality $f\circ h(x)=x$ on $\hat{\mathbb{C}}\backslash [-15,1]$ by definition.
		
		First we show $h(x)\neq -3$. Assume the contrary, then
		\begin{align*}
			x+1+(x+7)\cdot s\circ \mathfrak{a}(x) =-x-7.
		\end{align*}
		Since $x\neq -7$, we get
		\begin{align*}
			s\circ \mathfrak{a}(x) =-1
		\end{align*}
		which is not possible by the definition of $s$.
		
		For $x\neq \infty$, moving $\frac{x+1}{2}$ to the left side in the definition of $h$ we get:
		\begin{align*}
			\left( h(x)-\frac{(x+1)}{2}\right) ^2=\frac{(x+7)^2}{4}\cdot \left( 1-\frac{64}{(x+7)^2}\right).
		\end{align*}
		After computation it is
		\begin{align*}
			h(x)^2-h(x)+4&=x\cdot h(x)+3x\\
		\end{align*}
		Because $h(x)\neq -3$ we can state
		\begin{align*}
			\frac{h(x)^2-h(x)+4}{h(x)+3}=x \text{ thus } f\circ h(x)=x.
		\end{align*}
		For $x=\infty$, we know $f\circ h(\infty)=f(\infty)=\infty$.
		
		The equality $f\circ h(x)=x$ implies $h$ is injective on $\hat{\mathbb{C}}\backslash [-15,1]$, then we get $h\circ f(x)=x$ for $x$ on $h(\hat{\mathbb{C}}\backslash [-15,1])$.
		
		When $x \to \infty$, we can see that $s\circ \mathfrak{a}(x)$ approaches to $1$ and its Taylor series like
		\begin{align*}
			s\circ \mathfrak{a}(x)=1-\left( \frac{64}{2(x+7)^2}\right) -\left( \frac{64^2}{8(x+7)^4}\right) -higher\ order\ terms.
		\end{align*}
		Then we can say when $x \to \infty$, $h$ satisfies:
		\begin{align*}
			h(x) \longrightarrow \frac{x+1+x+7}{2}=x+4.
		\end{align*}
		This shows that $h$ is an inverse branch of $f$ satisfying the desired properties.
		
		Notice $-15$ and $1$ are the critial values of $f$ and they are not in the domain of $h$. By Implicit Function Theorem, $h$ is locally unique. Assume there is $\tilde{h}$ satisfying the same properties, then the set
		\begin{align*}
			\left\lbrace x \in \hat{\mathbb{C}}\backslash [-15,1]\mid h(x)=\tilde{h}(x)\right\rbrace 
		\end{align*} must be closed and open. Since $h(\infty)=\tilde{h}(\infty)=\infty$, we know $h$ and $\tilde{h}$ are the same on $\hat{\mathbb{C}}\backslash [-15,1]$.
	\end{proof}
	
	Now we are prepared to prove dynamical properties of $h$.
	
	\begin{proof}[Proof of Theorem \ref{iteration of h}]
		Define
		\begin{align}\label{def of e}
			\mathfrak{e}(x):=h(x)-x-4
		\end{align}
		to be the error term. For $\mathfrak{e}(x)$, we know for all $x \in \mathbb{H}_0$, by properties of $\mathfrak{a}$ in Lemma 5.2, the following estimate holds:
		\begin{align*}
			|\mathfrak{e}(x)|&=\left| \frac{x+1+(x+7)\cdot s\circ \mathfrak{a}(x) }{2}-x-4\right| \\
			&=\left| \frac{(x+7)\cdot s\circ \mathfrak{a}(x)-(x+7)}{2}\right| \\
			&=\left| \frac{-32}{(x+7)(s\circ \mathfrak{a}(x)+1)}\right| \\
			&<\frac{32}{|10+7|\left| 1+s\circ \mathfrak{a}(10) \right| } \\
			&=\frac{32}{17\left( 1+\frac{15}{17}\right) } \\
			&=1,
		\end{align*}
		which also implies
		\begin{align*}
			-1<\Re(\mathfrak{e}(x))<1.
		\end{align*}
		Then by \eqref{def of e} we have
		\begin{align*}
			\Re(h(x))>\Re(x)+4-1>13>10.
		\end{align*}
		So $h(x)$ also belongs to $\mathbb{H}_0$.
		
		Following the same idea, with $Re(h(x))>10$, we can replace $x$ by $h(x)$ in the estimate to have:
		\begin{align*}
			-1<\Re(\mathfrak{e}(h(x)))<1\text{ and }\Re(h^{\circ 2}(x))>\Re(h(x))+3>\Re(x)+6>10.
		\end{align*}
		
		Inductively, for every $x\in H_0$ and $n \in \mathbb{N^+}$, we have $Re(h^{\circ n}(x))>10$ thus $h^{\circ n}(x)$ is well defined and belongs to $\mathbb{H}_0$. The following estimates also hold:
		\begin{align}\label{real of h}
			\Re(h^{\circ n}(x))>\Re(x)+3n.
		\end{align}
		
		Through \eqref{def of e}, we get
		\begin{align*}
			h^{\circ n}(x)-x-4n=\sum_{j=1}^{n}(h^{\circ (j)}(x)-h^{\circ (j-1)}(x)-4)=\sum_{j=1}^{n}\mathfrak{e}\left( h^{\circ {j-1}}(x)\right) .
		\end{align*}
		With the help of Lemma \ref{def of a}, we can compute it like:
		\begin{align}\label{harmonic and h}
			\left| h^{\circ n}(x)-x-4n\right| &= \left| \sum_{j=0}^{n-1}\mathfrak{e}\left( h^{\circ j}(x)\right) \right| \notag \\
			\leq&\sum_{j=0}^{n-1}\left| \frac{32}{(1+s\circ \mathfrak{a} \circ h^{\circ j}(x))(h^{\circ j}(x)+7)}\right|  \notag \\
			&<\sum_{j=0}^{n-1}\frac{32}{\left( 1+\frac{15}{17}\right) (\Re(h^{\circ j}(x))+7)}  \notag \\
			&<\sum_{j=0}^{n-1}\frac{16}{\Re(x)+3j+7} .
		\end{align}
		
		Using the triangular inequality, we separate the \eqref{main estimate} into four. To make the first estimate, we deal with $\mathfrak{e}(x)$ more carefully. Enlightened by the form of \eqref{harmonic and h}, computing the sum between $\mathfrak{e}(x)$ and $\frac{16}{x+7}$:
		\begin{align*}
			\mathfrak{e}(x)+\frac{16}{x+7}=\frac{16\left( s\circ \mathfrak{a}(x)-1\right) }{(s\circ \mathfrak{a}(x)+1)(x+7)}=\frac{-16\cdot 64}{\left( s\circ \mathfrak{a}(x)+1\right)^2(x+7)^3},
		\end{align*}
		we have
		\begin{align*}
			\left| \mathfrak{e}(x)+\frac{16}{x+7}\right| <\frac{1024}{\left| x+7\right| ^3}.
		\end{align*}
		Then we can give the following:
		\begin{align}\label{c1}
			\left| h^{\circ n}(x)-x-4n+\sum_{j=0}^{n-1}\frac{16}{h^{\circ j}(x)+7}\right|&=\left| \sum_{j=0}^{n-1}\left( \mathfrak{e}(h^{\circ j}(x))+\frac{16}{h^{\circ j}(x)+7}\right) \right| \notag \\
			&<\sum_{j=0}^{n-1}\frac{1024}{\left| h^{\circ j}(x)+7\right| ^3} \notag\\
			&<\sum_{j=0}^{n-1}\frac{1024}{(\Re(x)+3j+7)^3}  \notag\\
			&<\frac{C_1}{\Re(x)^2}\text{ (by \eqref{cubic inequality} in Lemma \ref{harmonic lemma})},
		\end{align}
		where $C_1$ is a certain positive constant.
		
		For the second estimate, we compare $\frac{16}{h^{\circ j}(x)+7}$ and $\frac{16}{x+4j+7}$:
		\begin{align}\label{c2}
			\left| \sum_{j=0}^{n-1}\frac{16}{h^{\circ j}(x)+7}-\frac{16}{x+4j+7}\right| \leq &\sum_{j=0}^{n-1}\left| \frac{16\left( x+4j-h^{\circ j}(x)\right) }{\left( h^{\circ j}(x)+7\right)(x+4j+7) }\right|  \notag \\
			&<\sum_{j=1}^{n-1}\frac{32\log (\Re(x)+3j+7)}{3\left| h^{\circ j}(x)+7\right| \left| x+4j+7\right| }\text{ (by \eqref{harmonic and h} and \eqref{harmonic and log})} \notag \\
			&<\sum_{j=1}^{n-1}\frac{10\log (\Re(x)+3j+7)}{(\Re(x)+3j+7)^2}\text{ (by \eqref{real of h})} \notag \\
			&<\frac{C_2\log \Re(x)}{\Re(x)}\text{ (by \eqref{log quadratic inequality})},
		\end{align}
		in here $C_2$ is also some certain positive number.
		
		The third step is to apply \eqref{harmonic} and \eqref{quadratic inequality} in Lemma \ref{harmonic lemma}:
		\begin{align}\label{c3}
			\left| \sum_{j=0}^{n-1}\frac{16}{x+4j+7}-4\log \left( \frac{x+4n+7}{x+7}\right)  \right|<\sum_{j=0}^{n-1}\frac{64}{\left| x+4j+7\right| ^2}<\frac{C_3}{\Re(x)}.
		\end{align}
		Similiarly, $C_3$ is some certain postive number.
		
		About the last step, it is simply as:
		\begin{align}\label{c4}
			\left| 4\log \left( \frac{x+4n+7}{x+7}\right) -4\log \left( \frac{x+4n}{x}\right) \right| &<4\left| \log \left( 1+\frac{7}{x+4n}\right) \right| +4\left| \log \left( 1+\frac{7}{x}\right) \right| <\frac{C_4}{\Re (x)}
		\end{align}
		where $C_4>0$ is a certain number.
		
		Then we can choose our postive constant
		\begin{align*}
			C:=C_1+C_2+C_3+C_4,
		\end{align*}
		and get
		\begin{align*}
			\frac{C_1}{\Re(x)^2}+\frac{C_2 \log \Re(x)}{\Re(x)}+\frac{C_3}{\Re(x)}+\frac{C_4}{\Re (x)}<\frac{(C_1+C_2+C_3+C_4)\log \Re(x)}{\Re(x)}=\frac{C\log \Re(x)}{\Re(x)}
		\end{align*}
		since $\Re(x)>\log \Re(x)>1$ when $\Re (x)>10$.
		
		Combining three estimates \eqref{c1}, \eqref{c2} and \eqref{c3} by triangle inequality, we get the conclusion.
		
		Now for $\Re(h^{\circ n}(x))$ and $\Im(h^{\circ n}(x))$, we can give such bounds separately:
		\begin{align*}
			|\Re(h^{\circ n}(x))-\Re(x)-4n|<\sum_{j=0}^{n-1}\frac{16}{\Re(x)+3j+7}\ and\ |\Im(h^{\circ n}(x))-\Im(x)|<\sum_{j=0}^{n-1}\frac{16}{\Re(x)+3j+7}.
		\end{align*}
		
		From \eqref{harmonic and log}, the harmonic series can be controlled by:
		\begin{align*}
			\sum_{j=0}^{n-1}\frac{16}{\Re(x)+3j+7}<\frac{32}{3}\log(\Re(x)+3n+7)<10\log(\Re(x)+3n+7).
		\end{align*}
		
		For $\Re(h^{\circ n}(x))$ and $\Im(h^{\circ n}(x))$, now we have:
		\begin{align*}
			\Re(h^{\circ n}(x))>\Re(x)+4n-10\log(\Re(x)+3n+7)
		\end{align*}
		and
		\begin{align*}
			\Im(h^{\circ n}(x))<\Im(x)+10\log(\Re(x)+3n+7).
		\end{align*}
		Immediately we can compute the $\Im(\log(h^{\circ n}(x)))$ as:
		\begin{align*}
			0\leq \Im(\log(h^{\circ n}(x)))=\arctan\left( \frac{\Im(h^{\circ n}(x))}{\Re(h^{\circ n}(x))}\right) <\arctan \left( \frac{\Im(x)+10\log(\Re(x)+3n+7)}{\Re(x)+4n-10\log(\Re(x)+3n+7)}\right) 
		\end{align*}
		and for every $x \in \mathbb{H}_0$, the last one goes to $0$ when $n \to \infty$.
	\end{proof}
	
	\section{Universal Distance}
	
	In this section, we demonstrate that there is a positive distance between \(\supp(\mu_{\infty})\) (see \eqref{definition of mu} in Section 4) and \(\mathbb{R}^+\). Using this result, we will establish the analyticity of the pressure function and analyze its asymptotic behavior in the next section.
	
	\begin{theorem}\label{positive distance}
		The set $\supp(\mu_{\infty})$ is the union of $J(f)$, the Julia set of $f$ and all zeros of every polynomial in the sequence $\left\lbrace M_n(x) \right\rbrace _{n=0}^{\infty}$. Moreover,
		\begin{enumerate}
			\item[1.]There is $Q>0$ so that the distance between $\supp(\mu_{\infty})$ and $[\, 0, \infty)$ is bigger than $Q$.
			\item[2.]There is $R>0$ so that $\supp(\mu_{\infty}) $ is contained in the strip $[\, -3, \infty) \times [-R,R]$.
		\end{enumerate}
	\end{theorem}
	This shows a stronger statement than Theorem \ref{positive distance for zeros and real line}.
	
	In Section 6.1, we introduce the Fatou coordinate. Then, in Section 6.2, we demonstrate how to identify gaps between \(\supp(\mu_{\infty})\) and $\mathbb{R}^+$ under the new coordinate system.
	
	\subsection{Fatou Coordinate}
	
	Unlike the classical definition, we introduce our new coordinate in this way:
	\begin{proposition}\label{definition of F}
		Define
		\begin{align*}
			F_n(x)=h^{\circ n}(x)-4n+4\log(x+4n)
		\end{align*}
		for $x \in \mathbb{H}_0$ and $n \in \mathbb{N^+}$, then it converges uniformly to a univalent map $F$ when $n \to \infty$. In the mean time, $F$ also satisfies
		\begin{align}\label{translation of F}
			F\circ h(x)=F(x)+4
		\end{align}and
		\begin{align}\label{estimate of F}
			|F(x)-x-4\log x|\leq \frac{C \log \Re(x)}{\Re(x)}.
		\end{align}
		for all $x \in \mathbb{H}_0$.
		
		Moreover, $F$ maps real numbers to real numbers.
	\end{proposition}
	
	For the injection of $F$, we need the help of the following lemma.
	
	\begin{lemma}\label{injection lemma}
		If $E$ is a holomorphic function defined on the right half plane $\mathbb{H}^+:=\left\lbrace x \in \mathbb{C}|\Re(x)>0\right\rbrace $ and there is a constant $C^{\prime}>0$ so that
		\begin{align}\label{condition of E}
			|E(x)-x-4\log x|\leq \frac{C^{\prime} \log \Re(x)}{\Re(x)}
		\end{align}
		holds for all $x \in \mathbb{H}_0$, then we can find $R>0$ so that for all $x_1, x_2 \in \mathbb{H}^+$ with $\Re(x_1),\Re(x_2)>R$, we have
		\begin{align*}
			E(x_1)\neq E(x_2).
		\end{align*}
	\end{lemma}
	
	\begin{proof}
		We pick any points $x_1, x_2 \in \mathbb{H}^+$ so that $\Re(x_1),\Re(x_2)>R^{\prime}+1$ for some positive number $R^{\prime}>0$ satisfying:
		\begin{align*}
			\frac{C^{\prime}\log R^{\prime}}{R^{\prime}}<\frac{1}{5}\text{ and }\left| 4\log\left( \frac{R^{\prime}+1}{R^{\prime}-1}\right) \right| <\frac{1}{5}.
		\end{align*}
		
		If $|x_1-x_2|>1$, we rewrite $E(x_1)-E(x_2)$ as:
		\begin{align*}
			E(x_1)-E(x_2)=(E(x_1)-x_1-4\log x_1)-(E(x_2)-x_2-4\log x_2)+(x_1-x_2)+4\log \frac{x_1}{x_2}.
		\end{align*}
		By assumption its first and second terms have norms smaller than $\frac{1}{5}$. For the last term, we estimate it:
		\begin{align*}
			4\left| \log \frac{x_1}{x_2}\right| =4\left| \int_{x_1}^{x_2}\frac{1}{s}ds\right| <\frac{2}{5}|x_1-x_2|.
		\end{align*}
		Applying \eqref{condition of E} and the triangular inequality, we get
		\begin{align*}
			\left| E(x_1)-E(x_2)\right| >|x_1-x_2|-\frac{2}{5}|x_1-x_2|-\frac{1}{5}-\frac{1}{5}=\frac{1}{5}>0
		\end{align*}
		so $E(x_1)-E(x_2)$ can't be zero.
		
		Otherwise, if $|x_1-x_2|<1$, through Cauchy integral, for $x$ on the line $(x_1,x_2)$ we get
		\begin{align*}
			E^{\prime}(x)-1=\frac{\mathrm{d} (E(x)-x-4\log x_0)}{\mathrm{d} x}=\frac{1}{2\pi i}\oint _{\partial B(x,1)}\frac{E(z)-z-4\log x_0}{(z-x)^2}\mathrm{d} z.
		\end{align*}
		where $x_0:=\frac{x_1+x_2}{2}$.
		Then
		\begin{align*}
			|E^{\prime}(x)-1|&<\max_{z\in B(x,1)}\left| E(z)-z-4\log x_0\right| \\
			\leq &\max_{z\in B(x,1)}\left| E(z)-z-4\log z\right| + \max_{z\in B(x,1)}\left| 4\log \frac{z}{x_0} \right|\\
			&<\frac{1}{5}+\frac{1}{5}.
		\end{align*}
		That implies
		\begin{align*}
			\Re(F^{\prime}(x))>1-\frac{2}{5}=\frac{3}{5}
		\end{align*}
		for $x$ on $(x_1,x_2)$. From the following integral
		\begin{align*}
			E(x_1)-E(x_2)=\frac{1}{x_2-x_1}\int_{0}^{1}E^{\prime}(x_1+t(x_2-x_1))\mathrm{d} t,
		\end{align*}
		we can see $E(x_1)-E(x_2)$ is the product of two nonzero complex numbers, meaning it is not zero.
		
		By taking $R:=R^{\prime}+1$ we proved our statement.
	\end{proof}
	
	Now we prove the property of $F$.
	
	\begin{proof}[Proof of Proposition \ref{definition of F}]
		For any $x \in \mathbb{H}_0$, $k,j \in \mathbb{N^+}$ and $k>j$, we compute the difference:
		\begin{align*}
			F_k(x)-F_j(x)=h^{\circ k}(x)-h^{\circ j}(x)-4(k-j)+4\log \left( \frac{x+4k}{x+4j}\right) .
		\end{align*}
		Through Theorem \ref{iteration of h}, we can replace $x$ by $h^{j}(x)$ and $n$ by $k-j$ in Formula \ref{main estimate} to get
		\begin{align*}
			\left| h^{\circ k}(x)-h^{\circ j}(x)-4(k-j)+4\log \left( \frac{h^{\circ j}(x)+4(k-j)}{h^{\circ j}(x)}\right) \right| < \frac{C \log \Re(h^{\circ j}(x))}{\Re(h^{\circ j}(x))}.
		\end{align*}
		By triangular inequality,
		\begin{align*}
			|F_k(x)-F_j(x)|&<\frac{C \log \Re(h^{\circ j}(x))}{\Re(h^{\circ j}(x))}+\left| 4\log  \left( \frac{x+4k}{x+4j}\right) -4\log \left( \frac{h^{\circ j}(x)+4(k-j)}{h^{\circ j}(x)}\right)\right|  \\
			&=\frac{C \log \Re(h^{\circ j}(x))}{\Re(h^{\circ j}(x))}+\left| 4\log  \left( \frac{x+4k}{h^{\circ j}(x)+4(k-j)}\right) -4\log \left( \frac{x+4j}{h^{\circ j}(x)}\right)\right|  .
		\end{align*}
		
		From \eqref{main estimate}, the following quotient satisfies
		\begin{align*}
			\left| \frac{x+4j}{h^{\circ j}(x)}-1\right| <\frac{4\left| \log \left( \frac{x+4j}{x}\right) \right|+\frac{C\log \Re(x)}{\Re(x)} }{\left| x+4j-4\log \left( \frac{x+4j}{x}\right)\right| -\frac{C\log \Re(x)}{\Re(x)}},
		\end{align*}
		so it approaches to $1$ uniformly on $\mathbb{H}_0$ when $j \to \infty$.
		
		Similarly, for the other quotient we have
		\begin{align*}
			\left| \frac{x+4k}{h^{\circ j}(x)+4(k-j)}-1\right| <\frac{4\left| \log \left( \frac{x+4j}{x}\right) \right|+\frac{C\log \Re(x)}{\Re(x)} }{\left| x+4k-4\log \left( \frac{x+4j}{x}\right)\right| -\frac{C\log \Re(x)}{\Re(x)} },
		\end{align*}
		which means it approaches to $1$ uniformly on $H_0$ with any $k>j$ when $j \to \infty$.
		
		Then for any small $\epsilon >0$, we can pick $n_1 \in \mathbb{N^+}$ so that for all $k>j>n_1$ and all $x\in \mathbb{H}_0$, the following estimate holds:
		\begin{align*}
			\left| \log  \left( \frac{x+4k}{h^{\circ j}(x)+4(k-j)}\right) -\log \left( \frac{x+4j}{h^{\circ j}(x)}\right) \right| < \frac{\epsilon}{2}.
		\end{align*}
		
		By \eqref{main estimate}, we also know $\Re(h^{\circ j}(x))$ approaches to $\infty$ when $j \to \infty$, indicating there is $n_2 \in \mathbb{N^+}$ so that for all $j>n_2$ and all $x\in H_0$, the expression
		\begin{align*}
			\frac{C \log \Re(h^{\circ j}(x))}{\Re(h^{\circ j}(x))}<\frac{\epsilon}{2}.
		\end{align*}
		Together we have
		\begin{align*}
			|F_k(x)-F_j(x)|<\epsilon
		\end{align*}
		for all $k>j>\max\left\lbrace n_1,n_2\right\rbrace $ and $x \in \mathbb{H}_0$. Then we prove the uniform convergence of $F_n$.
		
		For \eqref{translation of F}, from the equality:
		\begin{align*}
			F_n\circ h(x)&=h^{\circ (n+1)}(x)-4n+4\log (h(x)+4n)\\
			&=h^{\circ (n+1)}(x)-4(n+1)+4\log (x+4n)+4+4\log (\frac{h(x)+4n}{x+4n})\\
			&=F_{n+1}(x)+4+4\log (\frac{h(x)+4n}{x+4n})
		\end{align*}
		if we take $n$ to $\infty$ on both side, the last term goes to $0$ then we proved \eqref{translation of F}.
		
		Notice by taking $n\to \infty$ in \eqref{main estimate}, for all $x\in \mathbb{H}_0$, we have
		\begin{align*}
			|F(x)-x-4\log x|\leq \frac{C \log \Re(x)}{\Re(x)}.
		\end{align*} So \eqref{estimate of F} also holds.
		
		Now we show the injection of $F$. Assume there are $x_1\neq x_2$ so that $F(x_1)=F(x_2)$. Then we can use \eqref{translation of F} to get:
		\begin{align*}
			F(h(x_1))=F(x_1)+4=F(x_2)+4=F(h(x_2)).
		\end{align*}
		Inductively, given $F(x_1)=F(x_2)$, for all $n \in \mathbb{N^+}$ we have
		\begin{align*}
			F(h^{\circ n}(x_1))=F(h^{\circ n}(x_2)).
		\end{align*}
		
		Since $F$ satisfies \eqref{estimate of F}, by Lemma \ref{injection lemma} we know $F(x-10)$ is injective on $\left\lbrace x\in \mathbb{H}^+| \Re(x)>\tilde{R} \right\rbrace $ for some $\tilde{R}>0$. According to \eqref{main estimate} and the injection of $h$, we can substitute $x_1$ and $x_2$ with $h^{\circ n}(x_1)$ and $h^{\circ n}(x_2)$ respectively for some $n\in \mathbb{N^+}$ to get $F$ is injective on $\mathbb{H}_0$.
		
		Since $h$ and $\log$ map real numbers to real numbers, $H$ would share the same property.
	\end{proof}
	
	\subsection{Appropriate Domains near the Critial Point}
	
	\begin{lemma}
		Define $\mathbb{H}_{-3}:=\left\lbrace x \in \mathbb{C}| \Re(x)<-3\right\rbrace $. We state that
		\begin{align*}
			\mathbb{H}_{-3} \subset F(f).
		\end{align*}
	\end{lemma}
	
	\begin{proof}
		We know
		\begin{align*}
			f(x)=\frac{x^2-x+4}{x+3}=x-4+\frac{16}{x+3}.
		\end{align*}
		If $\Re(x)<-3$, then
		\begin{align*}
			\Re(f(x))<\Re(x)-4<-7.
		\end{align*}
		So $f^{\circ n}(x)$ is attracted to $\infty$, implying $\mathbb{H}_{-3} \subset F(f).$
	\end{proof}
	
	\begin{lemma}
		We state that $\supp(\mu_{\infty})=J(f)\cup \mathfrak{P}$, where $\mathfrak{P}$ represents all preimages of $-1$ under $f$ and its iterations.
	\end{lemma}
	
	\begin{proof}
		We prove $\mathfrak{P}^{\prime}$, the derived set of $\mathfrak{P}$ equals $J(f)$.
	\end{proof}
		
	\begin{theorem}(Brolin)
		For $x^2+c$, assume $c \neq 0$, there is a probability measure $\nu$ whose support is the Julia set with the property that for any point $w\in \mathbb{C}$ with at most two exceptions, the preimage sets $f^{-n}(w)$ equidistribute to $\nu$.
	\end{theorem}
	
	\begin{lemma}
		Define $L_a:=\left\lbrace x \in \mathbb{C}|\Re(x)=a\right\rbrace $ for $a \in \mathbb{R}$. We have $\supp(\mu_{\infty})\cap \left\lbrace x\in \mathbb{C}|11<\Re(x)<\Re(f(L_{11}))\right\rbrace $ is bounded.
	\end{lemma}
	
	\begin{lemma}
		Define $l$ to be the function so that $F\circ h=l\circ F$ on $\mathbb{H}_0$. Then $l$ maps any vertical lines $L_a  \subset \mathbb{H}_0$ to $L_{a+4}$.
	\end{lemma}
	
	\begin{definition}
		For every $n \in \mathbb{N}^+$, define $\mathfrak{P}_n:=\supp(\mu_{\infty})\cap \left\lbrace x \in \mathbb{C}|\Re(x)>f^{n-1}(L_{11})\right\rbrace $.
		
		For $n=0$ define $\mathfrak{P}_0:=\supp(\mu_{\infty})\backslash \cup_{n=1}^{\infty}\mathfrak{P}_n$.
	\end{definition}
		
	\begin{lemma}\label{location of Cn}
		There is $N_1 \in \mathbb{N^+}$ so that $\mathfrak{P}_{N_1} \in \mathbb{H}_0$.
	\end{lemma}
	
	We initiate our analysis by studying the toy model $$\chi(w):=w^2.$$ There is a conformal isomorphism $\eta$ which maps the attracting basin of $g$ at $\infty$ to the one of $g$. Since $g$ is a two-to-one map, it has only two Fatou component. From Carleson and Gamelin's book [3] we know its Julia set is a Jordan curve. By Carathéodory's theorem we can extend $\eta$ on the boundaries of their basins as a homeomorphism.
	
	For this simplest case, by Theorem \ref{zeta converge} we need to focus on preimages of $\eta (-2) $ under $g$. We know $\eta(-2)$ locates outside of the filled Julia set(closed unit disk) of $g$, so the same for all its preimages under $g$. In that sense we can employ the counting result of preimages under $g$ to $g$.
	
	In the $w$-coordinate, g maps $(\, -\infty,-1]$ onto $[\, 1,\infty)$ and the second interval is invariant under $g$. In the $x$-coordinate, we find $(\, -\infty,-1]$ and $[\, 0,\infty)$ have same properties under $g$. So
	\begin{align}\label{property of eta}
		\eta((\, -\infty,-1])=(\, -\infty,-1],\ \eta([\, 1,\infty))=[\, 0,\infty).
	\end{align}
	Then we have $\eta(-2)$ locates on $(-\infty,-1)$. After $n$ steps, its preimages have the formula $\left| \eta(-2)\right| ^{\frac{1}{2^n}}e^{\frac{i\pi k}{2^n}}$, where $k=\pm 1,\pm 3,\pm 5,……,\pm \left( 2^n-1\right) $ and $\left| \eta(-2)\right|$ is a positive number greater than $1$.
	
	For the toy model, it is convenient to choose annulus sectors like
	\begin{align*}
		\mathfrak{B}_n:=\exp\left( \left[ 0,|\log \eta(-2)|^{\frac{1}{2^n}}\right] \times \left[ -\frac{i \pi}{2^n},\frac{i \pi}{2^n}\right] \right) 
	\end{align*}for every $n\in \mathbb{N^+}$ and count numbers of preimages in those domains near $w=0$.
	
	It is clear that every $\mathfrak{B}_n$ contains exactly $2$ preimages of stage $n$. For stage $n+1$, no wander $\mathfrak{B}_n$ contains 4 preimages of it. Generally we can say that for stage $n+k$, $\mathfrak{B}_n$ contains $2^{k+1}$ preimages of that stage for $k\in \mathbb{N}^+$. In the same time, the domain $\mathfrak{B}_n\backslash \mathfrak{B}_{n+1}$ contains $2$ preimages of stage $n$, $2$ of stage $n+1$,$\dots$, $2^{k+1}-2^k$ of stage $n+k$.
	
	The other reason we pick such shape of $\mathfrak{B}_n$ is that for every $n \in \mathbb{N^+}$ we have:
	\begin{align}\label{chi act on Bn}
		\chi(\mathfrak{B}_{n+1})=\mathfrak{B}_n,\ \chi(\mathfrak{B}_{n+2}\backslash \mathfrak{B}_{n+1})=\mathfrak{B}_{n+1}\backslash \mathfrak{B}_n.
	\end{align}
	
	Since $\mathfrak{B}_n$ locate near $w=1$, by the definition of $s$ we get
	\begin{align}\label{s act on Bn}
		s(\mathfrak{B}_n)=\mathfrak{B}_{n+1},\ s(\mathfrak{B}_{n+1}\backslash \mathfrak{B}_n)=\mathfrak{B}_{n+2} \backslash \mathfrak{B}_{n+1}
	\end{align}
	and $m$ is a bijection from $\mathfrak{B}_n$ to $\mathfrak{B}_{n+1}$ for every $n \in \mathbb{N^+}$.
	
	When $w$ goes to $1$ from the exterior of unit circle, $z$ goes to $0$ from the exterior of the cauliflower and then $x$ would go to infinity by $x=\phi^{-1}(z)=\frac{z+4}{z}$.
	
	Recall from Section 3 that
	\begin{align*}
		f(x)=\frac{x^2-x+4}{x+3}.
	\end{align*}
	By the commutative diagram
	\begin{align}\label{commutative diagram}
		\phi^{-1}\circ \eta^{-1}\circ \chi=f \circ \phi^{-1}\circ \eta^{-1},
	\end{align}
	for every $n \in \mathbb{N^+}$ if we definte domains as
	\begin{align*}
		\mathfrak{C}_n:=\phi^{-1}\circ \eta^{-1}\left( \mathfrak{B}_n \right)
	\end{align*}we would know they locate in the attracting basin of $\infty$ and
	\begin{align*}
		f(\mathfrak{C}_{n+1})=\mathfrak{C}_n \text{ and } f(\mathfrak{C}_{n+2}\backslash \mathfrak{C}_{n+1})=\mathfrak{C}_{n+1}\backslash \mathfrak{C}_n.
	\end{align*}
	
	Next we show $\mathfrak{C}_1 \subset \mathbb{C}\backslash [-15,1]$, which by definition implies $\mathfrak{C}_n \subset \mathbb{C}\backslash [-15,1]$ for all $n \in \mathbb{N^+}$. By \eqref{property of eta}, we know $(-\infty,-1)\cap \eta^{-1}(\mathfrak{B}_1)=\emptyset$. Since $[-1,-\frac{1}{4}]$ is contained in the filled Julia set of $g$, we also have $(\,-\infty,-\frac{1}{4}]\cap \eta^{-1}(\mathfrak{B}_1)=\emptyset$. Notice $\phi^{-1}((\,-\infty,-\frac{1}{4}])=[\,-15,1)$ and $\phi^{-1}(\infty)=1$, due to the bijection of $\phi$ on the whole Riemann sphere, it is straightforward that $\mathfrak{C}_1=\phi^{-1}\circ \eta^{-1}\left( \mathfrak{B}_1 \right)$ has no common elements with $[-15,1]$.
	
	Since $h$ sends $x$ to $\infty$ for $x$ with large norm, $\eta \circ \phi \circ h(x)$ would locate near $w=1$. In the commutative diagram $\phi^{-1}\circ \eta^{-1}\circ \chi=f\circ \phi^{-1}\circ \eta^{-1}$, when we take inverse on both sides and choose $h$ as the branch of $f^{-1}$, the corresponding one for $\chi^{-1}$ would be $s$ so we have
	\begin{align*}
		s\circ \eta \circ \phi =\phi \circ \eta \circ h.
	\end{align*}
	That means $m$ and $h$ are also conjugated, then
	\begin{align}\label{h on Cn}
		h(\mathfrak{C}_n)=\mathfrak{C}_{n+1},\ h(\mathfrak{C}_{n+1}\backslash \mathfrak{C}_n)=\mathfrak{C}_{n+2}\backslash \mathfrak{C}_{n+1}
	\end{align}
	hold for every $n \in \mathbb{N^+}$.
	
	Further, we can also prove that $\mathfrak{C}_n$ is contained in $\mathbb{H}_0$ when $n$ is large.
	
	\begin{proof}
		Given that $m(\mathfrak{B}_n)=\mathfrak{B}_{n+1}$, and $\mathfrak{B}_n$ is approaching $1$ as $n \to \infty$, we get $\eta^{-1}(\mathfrak{B}_n)$ is trending toward $0$ and $\mathfrak{C}_n$ toward infinity. To be more specific, first we can pick $N_1\in \mathbb{N^+}$ so that all $z \in \eta^{-1}(\mathfrak{B}_{N_1})$ satisfy $|z|<\frac{1}{5}$. From classic result, we also know if $z \in \eta^{-1}(\mathfrak{B}_{N_1})$, then $\Re (z)>|\Im (z)|$. Through the following computation
		\begin{align*}
			\phi^{-1}(z)=\frac{4}{z}+1=(\Re (z)-\Im (z)i)\frac{4}{|z|^2}
		\end{align*}
		we get $\Re(\phi^{-1}(z))>\frac{4\Re (z)}{2(\Re (z))^2}>10$ for all $z \in \eta^{-1}(\mathfrak{B}_{N_1})$. By definition $\mathfrak{C}_{N_1}=\phi^{-1}\circ \eta^{-1}(\mathfrak{B}_{N_1})$ so $\mathfrak{C}_{N_1} \in \mathbb{H}_0$.
	\end{proof}
	
	Define
	\begin{align*}
		\mathfrak{R}_n:=\left( \mathfrak{C}_n\backslash\mathfrak{C}_{n+1}\right) \cap \supp(\mu_{\infty}),
	\end{align*}
	for every $n \in \mathbb{N^+}$ and
	\begin{align*}
		\mathfrak{R}_0:=\supp(\mu_{\infty})\backslash \cup_{n=1}^{\infty}\mathfrak{R}_n.
	\end{align*}
	We have every $\mathfrak{R}_n$ is bounded for $n \in \mathbb{N}$ and
	\begin{align}\label{partition of mu}
		\supp(\mu_{\infty})=\cup_{n=0}^{\infty}\mathfrak{R}_n.
	\end{align}
	
	Now we can prove the main theorem.
	\begin{proof}[Proof of Theorem \ref{positive distance}]
		By definition $\supp(\mu_{\infty})$ is closed so $\overline{\mathfrak{R}_n} \subset \supp(\mu_{\infty})$. For each element $k$ in $\overline{\mathfrak{R}_n}$ with $n \in \mathbb{N}$, either it is a preimage of the point
		\begin{align*}
			\phi ^{-1}(-2)=-1
		\end{align*}under iterations of $f$, or belongs to $J(f)$, the Julia set of $f$.
		
		For the first case, notice
		\begin{align*}
			f([\, 0, \infty))=\left[ \, \frac{4}{3}, \infty\right) 
		\end{align*}so $k$ does not locate on $[\, 0, \infty)$.
		
		For the second case, we know $(-\infty,-1)\cup (0,\infty)$ is contained in $F(g)$, the Fatou set of $g$. Notice that
		\begin{align*}
			\phi ^{-1}((\, -\infty,-4] \cup (0,\infty) \cup \left\lbrace \infty \right\rbrace )=[\, 0,\infty)
		\end{align*}
		so $[\, 0,\infty)$ is also contained in $F(f)$, the Fatou set of $f$.
		
		In summary we can state
		\begin{align}\label{empty intersection}
			\supp (\mu_{\infty})\cap [\, 0,\infty)=\varnothing.
		\end{align}
		
		Next we show there is a positive gap. After changing coordinate under $F$, we can see that
		\begin{align*}
			F(\mathfrak{C}_{n+1})=F\circ h(\mathfrak{C}_n)=F(\mathfrak{C}_n)+4.
		\end{align*}
		Define
		\begin{align*}
			\mathfrak{D}_n=F(\mathfrak{C}_n)
		\end{align*}
		for every $n\in \mathbb{N^+}$, we would know all $\mathfrak{D}_n$ have the same shape and
		\begin{align*}
			\mathfrak{D}_n+4=\mathfrak{D}_{n+1}.
		\end{align*}
		
		By the definition of $\mathfrak{R}_n$, we also have
		\begin{align*}
			h(\mathfrak{R}_n)=\mathfrak{R}_{n+1}
		\end{align*}
		for every $n\in \mathbb{N^+}$. Similarly,
		\begin{align*}
			F(\mathfrak{R}_{n+1})=F(\mathfrak{R}_n)+4.
		\end{align*}
		
		By \eqref{estimate of F}, the map $F$ is close to the identity map when the variable has large real parts. Since $F$ is also a real function on $\mathbb{R}$, we get $F([\, 11,\infty))=[\, F(11),\infty)$. Because $\mathfrak{R}_{N_1}$ is bounded, we know $\overline{\mathfrak{R}_{N_1}}$ and $H(\overline{\mathfrak{R}_{N_1}})$ are compact. By \eqref{empty intersection} and compactness, there is $Q^{\prime}>0$ so that $dist(F(\overline{\mathfrak{R}_{N_1}}),[\, F(11),\infty))>Q^{\prime}$. Since all $F(\overline{\mathfrak{R}_n})$ have the same shape, we get
		\begin{align*}
			dist(F(\overline{\mathfrak{R}_n}),[\, F(11),\infty))>Q^{\prime}
		\end{align*}
		for all $n>N_1$.
		
		Again by \eqref{estimate of F}, $\Im(F(x)-x)$ can be estimated as:
		\begin{align*}
			\Im (F(x)-x)\leq \Im (4\log x)+\frac{C\log\Re(x)}{\Re(x)}.
		\end{align*}
		Due to \eqref{main estimate} in Theorem \ref{iteration of h}, we can take big enough $N_2 >N_1$ so that all $x \in \mathfrak{C}_n$ with $n \geq N_2$ satisfy:
		\begin{align*}
			\Im (4\log x)+\frac{C\log\Re(x)}{\Re(x)}<\frac{Q^{\prime}}{2}.
		\end{align*}
		Then for all $x \in \mathfrak{R}_n$ with $n > N_2$, we have
		\begin{align*}
			\Im(x) > Q^{\prime}-\frac{Q^{\prime}}{2}=\frac{Q^{\prime}}{2},
		\end{align*}
		which is equivalent to say for all $n > N_2$,
		\begin{align*}
			dist(\overline{\mathfrak{R}_n},[\, 11,\infty))>\frac{Q^{\prime}}{2}.
		\end{align*}
		Following the same way as in Lemma \ref{location of Cn}, we can find $N_3 > N_2$ so that every point $s$ in $\mathfrak{R}_{N_3}$ satisfy $\Re (s)>11$. Then it is clear that for all $n > N_3$, we have
		\begin{align*}
			dist(\overline{\mathfrak{R}_n},[\, 0,\infty))>\frac{Q^{\prime}}{2}.
		\end{align*}
		
		For $0\leq n\leq N_3$, by compactness there is $Q^{\prime \prime}>0$ so that
		\begin{align*}
			dist(\bigcup_{n=1}^{N_3}\overline{\mathfrak{R}_n},[\, 0,\infty))>Q^{\prime \prime}
		\end{align*}
		Define $Q=\min \left\lbrace \frac{Q^{\prime}}{2}, Q^{\prime \prime}\right\rbrace $, then
		\begin{align*}
			dist(\supp (\mu_{\infty}),[\, 0,\infty))>Q.
		\end{align*}
		
		To prove item 2, notice that
		\begin{align*}
			F(\mathfrak{C}_n)=F(h^{\circ (n-1)}(\mathfrak{C}_1))=F(\mathfrak{C}_1)+4(n-1).
		\end{align*}
		Again, according to \eqref{estimate of F}, for every $x\in \mathfrak{C}_n$ with $n>N_1$ we get
		\begin{align*}
			\Im(x) \leq \Im(F(x))-\Im(4\log x)+\frac{C\log\Re(x)}{\Re(x)}.
		\end{align*}
		Recall the definition of $\mathfrak{D}_n$, for all $n \geq N_2$, we have
		\begin{align*}
			\max_{x\in \mathfrak{C}_n}\Im(x) \leq \max_{x\in \mathfrak{D}_n}\Im(x)+\frac{Q^{\prime}}{2}
		\end{align*}
		With the fact $\mathfrak{D}_n+4=\mathfrak{D}_{n+1}$ and the compactness argument for $0\leq n\leq N_2$, we proved item 2.
	\end{proof}
	
	\section{The Pressure Function}
	
	For the last section, we compute the pressure function $p$.
	
	First, we introduce
	\begin{align*}
		m:=\lim_{n \to \infty}\frac{1}{3^n}\log|M_n|,
	\end{align*}the pressure-like function corresponding to $M_n$. With Corollary \ref{mu converge}, we have
	\begin{align*}
		m=\int_{\mathbb{C}}\log|y-t|d\mu_{\infty}(t).
	\end{align*}Using the partition of $\supp (\mu_{\infty})$ in \eqref{partition of mu}, we compute
	\begin{align*}
		m(x)=\sum_{n=0}^{\infty}\int_{\mathfrak{R}_n}\log|x-t|d\mu_{\infty}(t).
	\end{align*}
	Then
	\begin{align*}
		|m(x)-\log x|=\sum_{n=0}^{\infty}\int_{\mathfrak{R}_n}\log\left| 1-\frac{t}{x}\right| d\mu_{\infty}(t).
	\end{align*}
	
	Before further computation, we state some properties of $\mathfrak{R}_n$. Since $-1 \in \mathfrak{R}_0$ with weight $\frac{1}{3}$, we have
	\begin{align*}
		\mu_{\infty}(\mathfrak{R}_0)>\frac{1}{3}.
	\end{align*}
	Then
	\begin{align*}
		\mu_{\infty}(\cup_{n=1}^{\infty}\mathfrak{R}_n)<\frac{2}{3}
	\end{align*}
	In the mean time, from the fact that $h(\mathfrak{R}_n)=\mathfrak{R}_{n+1}$, we get
	\begin{align*}
		\mu_{\infty}(\mathfrak{R}_n)=3\mu_{\infty}(\mathfrak{R}_{n+1}).
	\end{align*}
	So the estimate
	\begin{align*}
		\mu_{\infty}(\mathfrak{R}_n)<\frac{4}{3^{n+2}}
	\end{align*} holds for every $n \in \mathbb{N^+}$.
	
	For any small $0<\epsilon<1$, we pick $N^{\prime}>0$ so that
	\begin{align*}
		\frac{3N^{\prime}}{\epsilon}>10 \text{ and } \frac{4}{3^{\frac{3N^{\prime}}{5\epsilon}+2}}<\frac{\epsilon}{6\log |Q|}
	\end{align*} hold. Then for any $x>\frac{3N^{\prime}}{\epsilon^2}$, we seperate $\supp(\mu_{\infty})$ into four parts:
	\begin{align*}
		\supp(\mu_{\infty})_1&:=\supp(\mu_{\infty})\cap \left\lbrace t\in \mathbb{C}|\Re(t)<x\cdot \epsilon \right\rbrace \\
		\supp(\mu_{\infty})_2&:=\supp(\mu_{\infty})\cap \left\lbrace t\in \mathbb{C}|x\cdot \epsilon\leq\Re(t)<x-1 \right\rbrace \\
		\supp(\mu_{\infty})_3&:=\supp(\mu_{\infty})\cap \left\lbrace t\in \mathbb{C}|x-1\leq \Re(t)<x+1\right\rbrace \\
		\supp(\mu_{\infty})_4&:=\supp(\mu_{\infty})\cap \left\lbrace t\in \mathbb{C}|x+1\leq \Re(t)\cdot \epsilon \right\rbrace 
	\end{align*}
	For the first set, we estimate the integral
	\begin{align*}
		\int_{\supp(\mu_{\infty})_1}\log\left| 1-\frac{t}{x}\right| d\mu_{\infty}(t)<\log \left| 1-\frac{\epsilon}{3}\right| <\frac{\epsilon}{3}.
	\end{align*}
	
	For $\mathfrak{R}_n$ contained in $\supp(\mu_{\infty})_2$ and $\supp(\mu_{\infty})_4$, we know
	\begin{align*}
		\frac{\log |x-t-4|-\log|x|}{\log |x-t|-\log |x|}<\frac{t+4}{t}<\frac{3}{2}
	\end{align*}So for $\mathfrak{R}_n$ contained in $\supp(\mu_{\infty})_2$ and $\supp(\mu_{\infty})_4$
	\begin{align*}
		\int_{\mathfrak{R}_n}\log\left| 1-\frac{t}{x}\right| d\mu_{\infty}(t)<\frac{1}{2}\int_{\mathfrak{R}_{n+1}}\log\left| 1-\frac{t}{x}\right| d\mu_{\infty}(t).
	\end{align*} For such $\mathfrak{R}_n$, the index $n$ must be bigger than $\frac{3N^{\prime}}{5\epsilon}$ so for all such $\mathfrak{R}_n$, their weight is controled by $\frac{4}{3^{\frac{3N^{\prime}}{5\epsilon}+2}}$. The integral on those two sets can be controled by
	\begin{align*}
		\int_{\supp(\mu_{\infty})_2}\log\left| 1-\frac{t}{x}\right| d\mu_{\infty}(t)+\int_{\supp(\mu_{\infty})_4}\log\left| 1-\frac{t}{x}\right| d\mu_{\infty}(t)<2\log |Q|\frac{4}{3^{\frac{3N^{\prime}}{5\epsilon}+2}}<\frac{\epsilon}{3}.
	\end{align*}
	
	For $\mathfrak{R}_n$ contained in $\supp(\mu_{\infty})_3$, the number of such $\mathfrak{R}_n$ is no greater than $1$. The weight must be smaller than $\frac{4}{3^{\frac{3N^{\prime}}{5\epsilon}+2}}$. Then the integral on the third set can be estimated like
	\begin{align*}
		\int_{\supp(\mu_{\infty})_3}\log\left| 1-\frac{t}{x}\right| d\mu_{\infty}(t)<\frac{4}{3^{\frac{3N^{\prime}}{5\epsilon}+2}}\log |Q|<\frac{\epsilon}{3}.
	\end{align*}
	
	Combining all estimates we get
	\begin{align*}
		|m(x)-\log x|<\epsilon.
	\end{align*}
	
	Then we are prepared to control the pressure function $p$.
	
	\begin{proof}[Proof of Theorem \ref{pressure function and log}]
		By definition of $p$ and \eqref{Zn and Mn}, for $y \in \mathbb{R}^+$ we get
		\begin{align*}
			p(y)&=\lim_{n \to \infty}\frac{1}{4\cdot 3^n}\log|Z_n(y)|\\
			&=\lim_{n \to \infty}\frac{1}{4\cdot 3^n}\log\left| \frac{2M_n(y^4)}{y^{3^n}}\right| \\
			&=\lim_{n \to \infty}\frac{1}{4\cdot 3^n}\log|2M_n(y^4)|-\frac{1}{4}\log |y|\\
			&=m(y)-\frac{1}{4}\log |y|.
		\end{align*}
		From the previous discussion we know $m(y)$ approaches to $\log y$ when $y$ goes to $\infty$, so
		\begin{align*}
			\lim_{y \to \infty}\left| p(y)-\frac{3}{4}\log y\right| =0
		\end{align*}
	\end{proof}


\begin{thebibliography}{99}
		\bibitem{ref1}T.D. Lee, C.N. Yang, Statistical theory of equations of state and phase transitions: II.Lattice gas and Ising model, Phys. Rev. 87 (1952) 410–419.
		\bibitem{ref2}Raffaella Burioni et al 1999 J. Phys. A: Wath. Gen. 32 5017
		\bibitem{ref3}Carleson, L., Gamelin, T. W. (1993). Complex Dynamics. Springer. 
	\end{thebibliography}
\end{document}